\documentclass{amsart}

\usepackage{xcolor}
\usepackage{mathtools}
\usepackage{amsmath}
\usepackage{amsthm}
\usepackage{graphicx}
\usepackage{amssymb}
\usepackage{epstopdf}
\usepackage{nicefrac}
\usepackage{mathrsfs}
\usepackage{hyperref}
\usepackage{enumitem}
\usepackage{tikz}
\usepackage{caption}
\usepackage{subcaption}
\usepackage{cite}
\usetikzlibrary{calc,arrows.meta,decorations.pathreplacing}

\newcommand{\R}{\mathbb R}

\theoremstyle{plain}
\newtheorem{theo}{Theorem}[section]
\newtheorem{lm}[theo]{Lemma}
\newtheorem{cor}[theo]{Corollary}
\newtheorem{prop}[theo]{Proposition}

\newtheorem{remark}{Remark}

\newtheorem*{example*}{Example}

\DeclareMathOperator{\conv}{conv}

\DeclareMathOperator{\Ha}{H}

\DeclareMathOperator{\PW}{PW}
\DeclareMathOperator{\supp}{supp}

\DeclareMathOperator{\Real}{Re}

\DeclareMathOperator{\Image}{Im}
\DeclareMathOperator{\cone}{cone}

\title{Boundary-Weighted Fourier Inequalities for Convex Domains}
\author[K. Bampouras]{Konstantinos Bampouras}
\address{Department of Mathematics and Systems Analysis, School of Science, Aalto University, Espoo, Finland}
\email{konstantinos.bampouras@aalto.fi}

\author[K.-M. Perfekt]{Karl-Mikael Perfekt}
\address{Department of Mathematical Sciences, Norwegian University of Science and Technology (NTNU), 7491 Trondheim, Norway}
\email{karl-mikael.perfekt@ntnu.no}

\date{\today}

\begin{document}
	
	\begin{abstract}
		We consider the natural family of Fourier inequalities for the Paley--Wiener space $\PW^q(\Omega)$, consisting of $L^q$-functions with Fourier support in a convex set $\Omega \subset \mathbb{R}^n$, $n \geq 2$, free of affine lines. Namely, 
		\begin{equation*}
			\int_{\Omega}\dfrac{|\hat{f}(x)|^p}{\omega_{\Omega}^d(x)}dx\leq C\|f\|_{L^q}^p,\quad f\in \PW^q(\Omega).
		\end{equation*} 
		Here $\hat{f}$ is the Fourier transform of $f$, $1 \leq p, q < \infty$, $d \in \R$, and $\omega_\Omega$ is the  frequency multiplier weight associated with the Paley--Wiener space of $\Omega$,
		\[
		\omega_{\Omega}(x)=m(\Omega\cap (2x-\Omega)), \qquad x \in \Omega.
		\]
		For an arbitrary polyhedron $P$, we completely characterize the triples $(p,q,d)$ which yield valid Fourier inequalities. For a ball $B$, we characterize the valid triples when $p \geq 2$. When $p < 2$, the situation is different for the ball, and natural critical inequalities fail. However, we show that the spherical restriction conjecture implies a family of subcritical Fourier inequalities for the ball, which in turn imply the Kakeya conjecture (in its Minkowski-form). Finally, we link our family of Fourier inequalities to the theory of truncated Hankel operators acting on the Paley--Wiener space of $\Omega$.
	\end{abstract}
	
	\maketitle

	\tableofcontents

	\section{Introduction and main results}
	
	The purpose of this article is to study Fourier inequalities for Paley--Wiener spaces 
	\[
	\PW^q(\Omega) = \{f \in L^q(\mathbb{R}^n) \, : \, \supp \hat{f} \subset \overline{\Omega}\}, \qquad 1 \leq q < \infty,
	\]
	consisting of functions with (distributional) Fourier support in some convex set $\overline{\Omega} \subset \R^n$, $n \geq 2$. Our theory treats open convex sets that do not contain any affine lines, and $\Omega$ will henceforth denote such a set. 
	
	The natural frequency multiplier to consider for Fourier inequalities on $\PW^q(\Omega)$ has been identified in \cite{bampouras2024besovspacesschattenclass}, in connection with the construction of Besov spaces $B^s_{p,q}(\Omega)$ of Paley--Wiener type. Namely, let
	\[
	\omega_{\Omega}(x)=m(\Omega\cap (2x-\Omega)), \qquad x \in \Omega,
	\]
	finite because $\Omega$ is free of lines. Powers $\omega_\Omega^{-s}$ of this Fourier weight play a similar role to the Fourier multiplier weights defining the Riesz and Bessel potentials in the classical theory; this is evidenced in particular by the fact that the corresponding scale of Besov spaces describes the Schatten classes for truncated operators of convolution-type on $\PW^2(\Omega)$, see \cite{bampouras2024besovspacesschattenclass}.
	
	The Fourier inequalities that we will consider are of the form
	\begin{equation}\label{hardy} 
		\int_{\Omega}\dfrac{|\hat{f}(x)|^p}{\omega_{\Omega}^d(x)}dx\leq C\|f\|_{L^q}^p,\quad f\in \PW^q(\Omega),
	\end{equation} 
	where $1 \leq p, q < \infty$ and $d \in \R$. We will refer to inequality \eqref{hardy} as the $(p,q,d)$ boundary-weighted Fourier inequality for $\Omega$. Rather unusually, we denote the Fourier-variable by $x \in \R^n$, to emphasize that our analysis will be almost entirely based on real-variable techniques on the Fourier side. 
	
	Two classical inequalities motivate important endpoints of this scale of boundary-weighted Fourier inequalities.
	The triple $(p,q,d)=(2,1,1)$ yields what we call Helson's inequality,
	\begin{equation} \label{eq:helsonineq}
		\int_\Omega \dfrac{|\hat{f}(x)|^2}{\omega_{\Omega}(x)}dx\leq C\|f\|_{L^1}^2, \quad f\in\PW^{1}(\Omega).
	\end{equation}
	When $n = 1$ and $\Omega = \R_+ = (0,\infty)$, the space $\PW^q(\R_+)$ is the usual analytic Hardy space $H^q$ of the upper half plane. In this case, \eqref{eq:helsonineq} with $C = 1/2$ is a classical inequality due to Carleman, which presents a strengthening of the isoperimetric inequality \cite{MR1984405}. When $\Omega = \R_+^n$, inequality \eqref{eq:helsonineq} was established by H. Helson \cite{MR2263964}, in order to show that every several variable Hankel operator of Hilbert--Schmidt class is generated by a bounded symbol, and that the corresponding symbol assignment is contractive.  The other prototypical endpoint is the $(1,1,1)$ Fourier inequality,
	\begin{equation} \label{eq:hardyineq}
		\int_\Omega \dfrac{|\hat{f}(x)|}{\omega_{\Omega}(x)}dx \leq C\|f\|_{L^1}, \quad f\in\PW^{1}(\Omega),
	\end{equation}
	which, when valid, is the Paley--Wiener analogue of Hardy's inequality for the analytic Hardy space $H^1$. Recall that Hardy's inequality for $H^1$ is the estimate
	\[
	\int_{\R_+} \frac{|\hat{f}(x)|}{x} \, dx \leq \pi \|f\|_{L^1}, \qquad f \in \PW^1(\R_+) = H^1.
	\]
	It is important not to confuse our type of Hardy inequality \eqref{eq:hardyineq} with Hardy inequalities for Sobolev spaces $W^{1,p}(\Omega)$. In our context, the functions $f$ have \textit{Fourier support} in $\Omega$. In the classical case, $n=1$, the Hardy and Helson inequalities are closely connected, essentially due to the availability of Nehari's theorem, see Lemma~\ref{lem:hardyineqandhilbert}. We will soon see that these endpoint inequalities can behave differently in higher dimensions.
	
	We first record a set of necessary conditions on the triple $(p,q,d)$, valid for every set $\Omega$. As usual, $p'$ and $q'$ denote the dual indices of $p$ and $q$, $\frac{1}{p} + \frac{1}{p'} = 1$, $\frac{1}{q} + \frac{1}{q'} = 1$.
	\begin{theo}
		\label{thm:intro-necessary}
		If the $(p,q,d)$ boundary-weighted Fourier inequality holds for $\Omega$, then
		\[
		q\leq2,\qquad p\leq q',
		\qquad d\leq d_c(p,q) = 1+\frac pq-p.
		\]
		When $1 \leq p < q < 2$, it is necessary that
		\begin{equation} \label{eq:defect}
			d < d_c(p,q).
		\end{equation}
		Furthermore, if \(\Omega\) is unbounded, then it must be that $p \geq q$ and $d=d_c(p,q)$.
	\end{theo}
	We shall additionally see, Theorem~\ref{plessqd=dc}, that when $1 \leq p < q < 2$, the defect presented by \eqref{eq:defect} can be made quantitative for sets $\Omega$ with positive affine surface area. 
	
	For polyhedra $P \subset \R^n$ and $q < 2$, it turns out that the necessary conditions of Theorem~\ref{thm:intro-necessary} are also sufficient. A key step is to prove, in Theorem~\ref{interpolation}, that we can actually interpolate between $\PW^1(P)$ and $\PW^p(P)$, $1 < p < \infty$, by lifting the corresponding statement from the product Hardy space.
	
	\begin{theo}\label{poly}
		Let \(P\subset\mathbb R^n\) be an open polyhedron which does not contain any affine line,
		and let \(1\leq p,q<\infty\).
		\begin{enumerate}[label=\textup{(\roman*)}]
			
			\item Suppose $1\leq q<2$ and $q\leq p\leq q'$. If $P$ is bounded,
			then the $(p,q,d)$  boundary-weighted Fourier inequality holds for $\Omega = P$ if and only if
			\[
			d\leq d_c(p,q) = 1+\frac pq-p.
			\]
			If $P$ is unbounded, it holds exactly when $d=d_c(p,q)$.
			
			\item Suppose that $1\leq p<q < 2$. If $P$ is bounded, then
			the $(p,q,d)$ inequality holds if and only if
			\[
			d<d_c(p,q).
			\]
			If $P$ is unbounded, the $(p,q,d)$ inequality fails for every $d$.
		\end{enumerate}
	\end{theo}
	For brevity, we have left out $q = 2$ from the statement. In this case, the problem is straightforward, only depending on the integrability of powers $\omega_\Omega^{-r}$, see Lemma~\ref{lem:qis2}.

        When $P = \R^n_+$, $\PW^q(\R^n_+)$ is nothing but the analytic product Hardy space. In this setting, on the diagonal $p=q < 2$, where the critical exponent is $d_c(q,q)=2-q$, the corresponding boundary-weighted Fourier inequality
        Theorem~\ref{poly} presents a Hardy--Littlewood-type inequality that has appeared in many different forms throughout the last hundred years. See for example \cite{DyachenkoEtAl2023}, which also includes an extension $0 < q < 1$.

	For a ball $B\subset\mathbb R^n$, $n \geq 2$, a simple computation shows that 
	\[
	\omega_B(x) \approx 
	\operatorname{dist}(x ,\partial B)^{(n+1)/2}.
	\]
	We begin our study of this case by writing down, more or less explicitly, an $L^\infty$-function whose distributional Fourier transform coincides with $\omega_B^{-1}$ on $B$; the existence of such a symbol directly implies that Helson's inequality \eqref{eq:helsonineq} holds for the ball, see Lemma~\ref{1/omega}. We do not know if it is possible to interpolate between $\PW^{p_1}(B)$ and $\PW^{p_2}(B)$, even for $1 < p_1 < p_2$. Indeed, a famous result of C. Fefferman \cite{MR296602} implies that there is no bounded projection $L^p(\R^n) \to \PW^p(B)$, unless $p = 2$, so a naive retraction argument cannot be applied. Nevertheless, we will devise an interpolation argument which gives the sharp positive result in the upper range $p \geq 2$.
	
	\begin{theo}\label{ball}
		Let $\Omega \subset\mathbb R^n$ be an open ball, where $n\geq2$, and suppose that
		\[
		1\leq q\leq2,
		\qquad 2\leq p<\infty.
		\]
		Then the $(p,q,d)$ boundary-weighted Fourier inequality holds for the ball if and only if
		\[
		p\leq q',
		\qquad d\leq d_c(p,q)=1+\frac pq-p.
		\]
	\end{theo}
	
	However, when $p < 2$, the curved boundary of the ball yields a completely different situation from that of the flat-faced polyhedra. In fact, in the critical diagonal case $(q,q, 2-q)$, the corresponding boundary-weighted Fourier inequality always
	fails. This is a consequence of the Besicovitch/Kakeya set construction of a measure zero set containing every direction, not dissimilar to the original disproof of the disc conjecture \cite{MR296602}, alluded to earlier.
	
	\begin{theo}\label{Hardynegate-intro}
		Let $\Omega = B\subset\mathbb R^n$ be a ball, $n\geq2$. For every
		\(1\leq q<2\), the $(q,q,d_c)$ inequality, 
		\[
		\int_B\frac{|\widehat f(x)|^q}{\omega_B(x)^{2-q}}\,dx
		\lesssim
		\|f\|_{L^q}^q,
		\qquad f\in\PW^q(B),
		\]
		fails.
	\end{theo}
	
	In fact, there is a direct link between the range of subcritical $(q,q,d)$ Fourier inequalities for the ball, $d < d_c(q,q) = 2-q$, and the Kakeya conjecture in its weaker Minkowski form. 
	
	\begin{theo} \label{thm:minkowskicons}
		Let \(1\le q<2\).  Assume that the $(q,q,d)$ boundary-weighted Fourier inequality for $B$ holds for every $d < d_c(q,q) = 2-q$, 
		\begin{equation*} 
			\int_{B}\dfrac{|\hat{f}(x)|^q}{\omega_{B}^d(x)}dx\leq C_{q, d}\|f\|_{L^q}^q,\quad f\in \PW^q(B). 
		\end{equation*}
		Then every compact Kakeya set in $\mathbb R^n$ has full Minkowski dimension $n$.
	\end{theo}
	We do not know whether the hypothesis of Theorem~\ref{thm:minkowskicons} is ever satisfied when $n \geq 3$. For $n = 2$, we will see below that the subcritical $(q,q,d)$ Fourier inequality actually does hold for the unit disc $\mathbb{D}$, when $\frac{4}{3} \leq q < 2$ and $d < 2-q$. The conclusion that every compact Kakeya set does have full Minkowski dimension is of course already known to be true for $n = 2$ \cite{Cordoba77}, and for $n=3$ by the very recent positive resolution of the Kakeya conjecture in three dimensions \cite{WangZahl2025Kakeya}; see also the streamlined proof given in \cite{GuthWangZahl2026}.
	
	There is also an implication in the opposite direction to be obtained from Fourier restriction theory. Recall that the spherical restriction conjecture \cite{Tao2004restriction} asserts the validity of the estimates
	\[
	\|\widehat F|_{S^{n-1}}\|_{L^p(S^{n-1})}
	\lesssim \|F\|_{L^q(\mathbb R^n)}, \qquad 1\le q< \frac{2n}{n+1}, \; \; 1\le p \le \frac{(n-1)q'}{n+1}.
	\]
	\begin{theo}
		Assume that the spherical restriction conjecture holds for some $n \geq 2$. Then, if $\frac{2n}{n+1} \leq q < 2$, the $(q, q, d)$ boundary-weighted Fourier inequality holds for the ball $B$ for every $d < d_c(q,q) = 2-q$. 
	\end{theo}
    The paradigm that restriction estimates imply weighted Fourier inequalities is not new, see \cite{DeCarliGorbachevTikhonov2017}. 
    
    It is of course already well known that the spherical restriction conjecture implies the Kakeya conjecture. In light of these two results, the scale of subcritical $(q,q,d)$ Fourier inequalities for $B$, $d < d_c$, arguably sits in between the restriction problem and the Kakeya conjecture. When $n=2$, the restriction conjecture for the circle is true, and we obtain a sharp unconditional corollary.
	\begin{cor} \label{cor:zygmundintro}
		The $(q,q,d)$ boundary-weighted Fourier inequality 
		\[
		\int_{\mathbb{D}} \frac{|\hat{f}(x)|^q}{(1-|x|)^{\frac{3}{2}d}} \, dx \lesssim \|f\|_{L^q}^q, \qquad f \in \PW^q(\mathbb{D}),
		\]
		holds for the unit disc $\mathbb{D} \subset \R^2$, when $\frac43 \leq q <2$ and $d < d_c(q,q) = 2-q$.  
	\end{cor}
	In the final section of the paper, we will return to our original motivation for studying the Helson inequality, namely, the study of (truncated) Hankel operators on $\PW^2(\Omega)$. We defer precise definitions to Section~\ref{Sec7}. As discussed earlier, we obtain the following corollary of Theorem~\ref{ball}, answering the open question posed in \cite{MR4502777}.
    \begin{cor} \label{cor:HSball}
    Every Hilbert-Schmidt Hankel operator on $\PW^2(B)$ is generated by a bounded symbol.
    \end{cor}
    Alternatively, we say that Nehari's theorem holds for $S^2(\PW^2(B))$. More generally, we are going to show that there is a direct link between Nehari's theorem for $S^p(\PW^2(\Omega))$ and the $(p', 1, 1)$ boundary-weighted Fourier inequality
	$$\int_\Omega\dfrac{|\hat{f}(x)|^{p'}}{\omega_\Omega(x)}dx\lesssim \|f\|_{L^1}^{p'},\quad f\in\PW^1(\Omega),$$
    see Propositions~\ref{firstprop} and \ref{prop:Nehari-implies-Hardy}.
	
	Before closing the introduction, it must be mentioned that we will actually begin the paper by proving a foundational geometric theorem. In order to define and apply the Besov spaces $B^s_{p,q}(\Omega)$ of Paley--Wiener type in \cite{bampouras2024besovspacesschattenclass}, it was necessary to introduce a notion of admissible sets $\Omega$, which required $\Omega$ to be decomposable into parallelepipeds satisfying certain natural axioms, see Section~\ref{Sec2}. We will complete the study by proving that this hypothesis is actually automatic.
	\begin{theo} \label{thm:main}
		Every open convex set \(\Omega\subset\mathbb R^n\), free of affine lines, is
		admissible. 
	\end{theo}
	We will repeatedly require this theorem in our study of Fourier inequalities.
	
	\subsection{Organization}
	Section \ref{Sec2} develops the admissibility theory for convex sets $\Omega$, and briefly introduces the associated Besov spaces of Paley--Wiener type. Section \ref{Sec3} studies boundary-weighted Fourier inequalities for general sets $\Omega$, in particular proving the necessary conditions stated in Theorem~\ref{thm:intro-necessary}. Section \ref{Sec4} is devoted to polyhedra, giving a complete characterization of the valid Fourier inequalities in this setting. In Section \ref{Sec5} we prove Theorem~\ref{ball} for the ball. Section~\ref{Sec6} explores the connections with the Kakeya conjecture and spherical restriction conjecture. Finally, in Section~\ref{Sec7} we consider the applications and links to Hankel operators on Paley--Wiener spaces.
	
	\subsection{Acknowledgments} The second author was supported by grant no. 334466 of the Research Council of Norway, "Fourier Methods and Multiplicative Analysis". The authors used ChatGPT 5.4/5.5/5.6 by OpenAI as a tool in the development of this paper. However, as a project started quite a while ago, it was primarily produced in the old-school way. The authors independently wrote and verified all parts of the article.
	
	%

	\section{The admissibility of convex sets and Besov spaces of Paley--Wiener type}\label{Sec2}
	\subsection{Admissibility}
	\par Let $\Omega\subset \mathbb{R}^n$ be an open convex set that does not contain affine lines. Then, for $x\in \Omega$, the set $\Omega\cap (2x-\Omega)$ is bounded and we can define the natural Fourier weight-function 
	$$\omega_{\Omega}(x)=m(\Omega\cap (2x-\Omega))<\infty.$$
	Accordingly, a framework for Besov spaces of Paley--Wiener type, consisting of distributions with Fourier support in $\Omega$, was introduced in \cite{bampouras2024besovspacesschattenclass}. In order for these Besov spaces to be useful in the study of convolution-type operators on $\PW(\Omega)$, the set $\Omega$ needs to satisfy a minimal set of geometric axioms. If $\Omega$ satisfies these axioms, it is called admissible (or $a-$admissible). In \cite{bampouras2024besovspacesschattenclass} only certain special cases were treated, but we will begin by showing that every set is actually admissible.
	\begin{theo}\label{admissibility}
		Every convex set $\Omega\subset \mathbb{R}^n$ that does not contain affine lines is $a$-admissible, for every $a>1$.
	\end{theo}
	
	The definition of admissibility is as follows. An open convex set is $a$-admissible, $a > 1$, if it admits a family of parallelepipeds $\{A_j^i\}$, the indices $(i,j)$ running through a subset of $\mathbb{N} \times \mathbb{Z}$, that satisfies the following axioms:
	\begin{enumerate}[label=\textbf{Axiom \arabic*.}, leftmargin=*]
		\item \label{Axiom1} There exists $M>0$ such that $A_j^i\subset \omega_{\Omega}^{-1}([a^{j-M},a^{j+M}])$.
		\item \label{Axiom2}There exist $c_1,c_2>0$ such that $$c_1a^j\leq m(A_j^i)\leq c_2 a^j.$$
		\item \label{Axiom3} There exists an $\epsilon\in (0,1/2)$ such that
		$$\Omega\subset \bigcup_{j,i}T_{j,i}^{-1}(-\frac{1}{2}+\epsilon,\frac{1}{2}-\epsilon)^n,$$
		where $T_{j,i}$ is the affine bijection that maps $A_j^i$ onto $(-\frac{1}{2},\frac{1}{2})^n$.
		\item \label{Axiom4} There exists $C>0$ such that
		$$T_{j,i}A_k^l\subset C(-\frac{1}{2},\frac{1}{2})^n,$$ whenever $A_j^i\cap A_k^l\neq\emptyset$.
		\item \label{Axiom5} The sets 
		$$J_{j,i,k,l}=\{(\beta,\gamma):A_\beta^\gamma \cap \frac{1}{2}(A_j^i+A_k^l)\neq \emptyset \},$$ are of bounded cardinality, uniformly for every choice of indices $j,i,k,l$.
	\end{enumerate}
	These axioms for the cover $\{A^i_j\}$ should be relatively self-explanatory. We note, however, that  Axiom 4 implies that adjacent parallelepipeds in the cover have similar orientation and size.
	
	To prove Theorem~\ref{admissibility} we first need to introduce some notation. For $x\in \Omega$ let $M(x)=\Omega\cap(2x-\Omega)$ be the Macbeath region at $x$ \cite{Macbeath1952}, the largest $x$-symmetric subset of $\Omega$ containing $x$, and
	let $M(x,c) =(1-c)x+cM(x)$ be the dilation of $M(x)$ by $c$ with respect to its center $x$. We will also need the following two lemmas.
	\begin{lm}\label{easy}
		Let $x,y\in \Omega$, $t\in (0,1)$ with $x\in M(y,t)$. Then $$(1-t)^n \omega_\Omega(y)\leq \omega_\Omega(x) \leq (1+t)^n\omega_\Omega(y).$$
	\end{lm}
	\begin{proof}
		Let $z\in M(y)$ such that $x=(1-t)y+tz$. By the concavity of $\omega_{\Omega}^{1/n}$ \cite[Lemma 2.1]{bampouras2024besovspacesschattenclass},
		$$\omega_{\Omega}(x)\geq \left((1-t)\omega^{1/n}_{\Omega}(y)+t \omega_{\Omega}^{1/n}(z)\right)^n\geq (1-t)^n\omega_{\Omega}(y).$$
		For the upper bound, let us observe that
		$$x=(1-t)y+tz= (1+t)y-t(2y-z).$$
		Again by concavity, since $2y-z\in \Omega$ we get that
		$$ \omega_{\Omega}(y)\geq (\frac{1}{1+t}\omega_{\Omega}^{1/n}(x)+\frac{t}{1+t}\omega_{\Omega}^{1/n}(2y-z))^{n}\geq (\frac{1}{1+t})^n\omega_{\Omega}(x).$$
		The proof is complete.
	\end{proof}
	The next key result is analogous to that of \cite[Lemma~6.2]{MR2327291}.
	\begin{lm}\label{Macbeath}
		Let $x,y\in \Omega$. For $t\in (0,1)$, if $$M(x,t)\cap M(y,t)\neq\emptyset,$$ then 
		$$M(x,t)\subset M(y,\frac{t(3+t)}{1-t}).$$
	\end{lm}
	\begin{proof}
		Take $w\in M(x,t)$, letting $a\in M(x)$ be such that $w=(1-t)x+ta$. Let $\Lambda=\frac{t(3+t)}{1-t}$, and let $e$ be such that $w=(1-\Lambda)y+\Lambda e$. To prove the lemma, it suffices to show that $e\in M(y)$. 
		
		Let $z\in M(x,t)\cap M(y,t)$, with associated $u\in M(x)$ and $ v\in M(y)$ such that $z=(1-t)x+tu=(1-t)y+tv$. Using the fact that $x=y+\frac{t}{1-t}(v-u)$, we compute
		\begin{eqnarray}
			e&=&\frac{1}{\Lambda}w-\frac{(1-\Lambda)}{\Lambda}y=\frac{1-t}{\Lambda}x+\frac{t}{\Lambda}a-\frac{(1-\Lambda)}{\Lambda}y \nonumber  \\
			&=&\frac{1-t}{\Lambda}(y+\frac{t}{1-t}(v-u))+\frac{t}{\Lambda}a-\frac{(1-\Lambda)}{\Lambda}y \nonumber \\
			&=&\frac{\Lambda-t}{\Lambda}y+\frac{t}{\Lambda}v-\frac{t}{\Lambda}u+\frac{t}{\Lambda}a.\nonumber
		\end{eqnarray}
		Now let us observe that 
		\begin{align} 
			\frac{1-t}{1+t}(2x-u)&-\frac{2(1-t)}{1+t}y-\frac{2t}{1+t}v \nonumber \\
			&= \frac{1-t}{1+t}(2x-u)-\frac{2(1-t)}{1+t}(x+\frac{t}{1-t}(u-v))-\frac{2t}{1+t}v=-u \nonumber
		\end{align} 
		Combining the above equalities we derive that
		\begin{eqnarray}
			e &=& \frac{\Lambda-t}{\Lambda}y+\frac{t}{\Lambda}v+\frac{t}{\Lambda}(\frac{1-t}{1+t}(2x-u)-\frac{2(1-t)}{1+t}y-\frac{2t}{1+t}v )+\frac{t}{\Lambda}a  \nonumber \\
			&=&\frac{\Lambda+\Lambda t-3t+t^2}{\Lambda(1+t)}y+\frac{t(1-t)}{\Lambda(1+t)}(2x-u) +\frac{t(1-t)}{\Lambda(1+t)}v+\frac{t}{\Lambda}a. \nonumber 
		\end{eqnarray}
		If we write the above expression as 
		$$e=Ay+B(2x-u)+Bv+Ca,$$ then we can notice that 
		$A,B,C>0$ and $A+2B+C=1$. Now since $y,2x-u,v,a\in \Omega$ we conclude that $e\in \Omega$.
		
		It remains to show that $2y-e\in \Omega$. From the expression
		$$e=\frac{\Lambda-t}{\Lambda}y+\frac{t}{\Lambda}v-\frac{t}{\Lambda}u+\frac{t}{\Lambda}a,$$
		we get
		$$2y-e=\frac{\Lambda+t}{\Lambda}y-\frac{t}{\Lambda}v+\frac{t}{\Lambda}u-\frac{t}{\Lambda}a .$$
		Using now the fact that $x=y+\frac{t}{1-t}(v-u)$ and that $\Lambda=\frac{t(3+t)}{1-t}$ we can compute
		\begin{eqnarray}
			2y-e&=&\frac{\Lambda+t}{\Lambda}y-\frac{t}{\Lambda}v+\frac{t}{\Lambda}u-\frac{t}{\Lambda}a \nonumber \\
			&=& 	\frac{4}{3+t}y-\frac{1-t}{3+t}v+\frac{1-t}{3+t}u-\frac{1-t}{3+t}a \nonumber \\
			&=&\frac{2(1+t)}{3+t}y-\frac{1+t}{3+t}v+\frac{2(1-t)}{3+t} \left(y+\frac{t}{1-t}(v-u)\right)-\frac{1-t}{3+t}a+\frac{1+t}{3+t}u\nonumber\\
			&=&\frac{2(1+t)}{3+t}y-\frac{1+t}{3+t}v+\frac{2(1-t)}{3+t}x-\frac{1-t}{3+t}a+\frac{1+t}{3+t}u\nonumber\\
			&=& \frac{1+t}{3+t}(2y-v)+\frac{1-t}{3+t}(2x-a)+	\frac{1+t}{3+t}u.	\nonumber
		\end{eqnarray}
		Since $\frac{1+t}{3+t},\frac{1-t}{3+t},\frac{1+t}{3+t}>0$, $\frac{1+t}{3+t}+\frac{1-t}{3+t}+\frac{1+t}{3+t}=1$ and $2y-v,2x-a,u\in\Omega$, the proof is complete.
	\end{proof}
	We now prove the main result of this subsection.
	\begin{proof}[Proof of Theorem \ref{admissibility}]
		We will apply John's theorem \cite{John1948}, which implies that for any open convex body $K$, symmetric around its center $x$, there is a dimension-dependent constant $C(n) > 1$ and a parallelepiped $P$ such that
		$$P \subset K \subset (1-C(n)) x  + C(n)P.$$ For the rest of the proof, we fix two constants $s$ and $\eta$ such that 
		$$s\eta<\frac{1}{5}, \quad s>4C(n).$$
		
		For each $j \in \mathbb{Z}$ such that $\Delta_j := \omega_\Omega^{-1}((a^{j-1}, a^{j+1}))$ is non-empty, we let
		$$X_j=\{x_j^i\in \Delta_j \, : \, i\in I_j\}$$
		be a maximal family such that the sets $M(x_j^i,\eta)$ are pairwise disjoint, where $I_j \subset \mathbb{N}$ is the corresponding index set. By John's Theorem, there is for each $x_j^i \in X_j$ a parallelepiped $P_j^i$ centered at $x_j^i$ such that $$P_j^i\subset M(x_j^i,\eta)\subset (1-C(n))x_j^i+C(n)P_j^i.$$
		We will show that 
		$$\{A_j^i\} := \{(1-s)x_j^i+sP_j^i\},$$
		is a cover of the desired type, by verifying each axiom in turn.
		
		\ref{Axiom1}
		Let us first observe that
		$$A_j^i\subset M(x_j^i,s\eta).$$
		By Lemma \ref{easy}, for $x\in A_j^i$,
		$$a^{j-1}(1-s\eta)^n\leq (1-s\eta)^n\omega_\Omega(x_j^i)\leq \omega_\Omega(x)\leq (1+s\eta)^n\omega_\Omega(x_j^i)\leq a^{j+1}(1+s\eta)^n.$$
		\par \ref{Axiom2}
		By definition,
		$$m(A_j^i)\approx m(P_j^i)\approx m(M(x_j^i,\eta))\approx m(M(x_j^i))\approx a^j.$$
		\par  \ref{Axiom3}
		Note that
		\begin{align*}
			M(x_j^i,4\eta) &=4M(x_j^i,\eta)-3x_j^i \\ &\subset (1-4C(n))x_j^i+4C(n)P_j^i=(1-\frac{4C(n)}{s})x_j^i+\frac{4C(n)}{s}A_j^i.
		\end{align*} 
		The right-hand side is the contraction of $A_j^i$ by $\frac{4C(n)}{s}<1,$ which coincides with $T_{j,i}^{-1}(-1/2+\epsilon,1/2-\epsilon)^n$ for some $\epsilon=C(n,s).$
		To verify the axiom, it thus suffices to show that $$\Delta_j\subset \bigcup_{i\in I_j}M(x_j^i,4\eta).$$
		
		Let $x\in \Delta_j$. Then, by the maximality of $X_j$, there is $l\in I_j$ such that $M(x,\eta)\cap M(x_j^l,\eta)\neq \emptyset$. Thus by Lemma \ref{Macbeath}, $$x\in M(x,\eta)\subset M(x_j^l,\frac{\eta(3+\eta)}{1-\eta})\subset M(x_j^l,4\eta),$$ where we used that $\eta<\frac{1}{5}$ in the final step.
		\par  \ref{Axiom4}
		Suppose that $A_j^i\cap A_k^l\neq \emptyset$. By definition, $$A_j^i\subset M(x_j^i,s\eta),\quad A_k^l\subset  M(x_k^l,s\eta),$$ and thus, by Lemma \ref{Macbeath} and the fact that $s\eta<1/5$, we get that
		$$A_k^l\subset M(x_k^l,s\eta)\subset M(x_j^i,4s\eta).$$
		Now taking $T_{j,i}$ and observing that it respects dilations with respect to the center $x_j^i$, we get that
		$$T_{j,i}A_k^l\subset T_{j,i}M(x_j^i,4s\eta)=4s T_{j,i}M(x_j^i,\eta)\subset 4sC(n)T_{j,i}P_j^i=4C(n)T_{j,i}A_j^i,$$
		as desired.
		\par \ref{Axiom5}
		Let us take $A_j^i,A_k^l$ and $A_\beta^\gamma$ with $A_\beta^\gamma \cap \frac{1}{2}(A_j^i+A_k^l)\neq \emptyset$, that is, $(\beta, \gamma) \in J_{j,i,k,l}$.
		Since $A_j^i\subset M(x_j^i,s\eta)$, a simple observation shows 
		$$\frac{1}{2}(A_j^i+A_k^l)\subset \frac{1}{2}(M(x_j^i,s\eta)+M(x_k^l,s\eta))\subset M(\frac{x_j^i+x_k^l}{2},s\eta).$$
		Therefore $M(x_\beta^\gamma,s\eta)\cap M(\frac{x_j^i+x_k^l}{2},s\eta)\neq \emptyset$,
		and thus, by Lemma \ref{Macbeath} again,
		$$M(x_\beta^\gamma,s\eta)\subset M(\frac{x_j^i+x_k^l}{2},4s\eta).$$
		By Lemma \ref{easy}, we deduce that $$\omega_{\Omega}(x_\beta^\gamma)\approx \omega_{\Omega}(\frac{x_j^i+x_k^l}{2}).$$
		Let $ J'_{j,i,k,l}=\{\beta: \text{ there is $\gamma$ such that } (\beta,\gamma)\in J_{j,i,k,l}\}$. For $(\beta,\gamma),(\beta',\gamma')\in J_{j,i,k,l}$ we have just seen that $\omega_{\Omega}(x_\beta^\gamma)\approx \omega_{\Omega}(x_{\beta'}^{\gamma'})$, and thus $|\beta - \beta'|$ is uniformly bounded across all such pairs. Therefore
		the sets $J'_{j,i,k,l}$ have uniformly bounded cardinality. Therefore
		\begin{eqnarray}
			\omega_{\Omega}(\frac{x_j^i+x_k^l}{2})&\approx& m(M(\frac{x_j^i+x_k^l}{2},4s\eta))
			\geq m(\bigcup_{(\beta,\gamma)\in J_{j,i,k,l}} M(x_\beta^\gamma,\eta) )  \nonumber \\
			&\geq& (\# J'_{j,i,k,l})^{-1} \sum_{\beta\in J'_{j,i,k,l}} m(\bigcup_{\gamma: (\beta,\gamma)\in J_{j,i,k,l}} M(x_\beta^\gamma,\eta) )\nonumber \\
			&=& (\# J'_{j,i,k,l})^{-1} \sum_{\beta\in J'_{j,i,k,l}}\sum_{\gamma: (\beta,\gamma)\in J_{j,i,k,l}}  m( M(x_\beta^\gamma,\eta) )\nonumber \\
			&\gtrsim& \sum_{\beta\in J'_{j,i,k,l}}\sum_{\gamma: (\beta,\gamma)\in J_{j,i,k,l}}  \omega_\Omega(x_\beta^\gamma)\approx \#J_{j,i,k,l} \omega_\Omega(\frac{x_j^i+x_k^l}{2}). \nonumber
		\end{eqnarray}  
		This demonstrates the final axiom and completes the proof.
	\end{proof}
	\subsection{Besov spaces} \label{subsec:besov}
	Having shown that every open convex set $\Omega$ not containing affine lines is admissible, we can employ the Besov spaces of Paley--Wiener type constructed in \cite{bampouras2024besovspacesschattenclass}. Let $\{A_j^i\}$ be a family of parallelepipeds satisfying \ref{Axiom1} -- \ref{Axiom5} Then, by \cite[Proposition 2.6]{bampouras2024besovspacesschattenclass}, we can find a family of smooth functions $\{\phi_j^i\}$ that satisfies
	\begin{enumerate}
		\item $\supp\widehat{\phi}_j^i\subset A_j^i$,
		\item $\sum_{j,i}\widehat{\phi}_j^i=\chi_{\Omega}$ and
		\item $\sup_{j,i}\|\phi_j^i\|_{L^1}<\infty$.
	\end{enumerate}
	We will denote the collection of all such Fourier-partitions by $\mathscr{P}(A_j^i)$. 
	
	Let $1\leq p,q\leq \infty$ and $s\in \mathbb{R}$. The Besov spaces 
	we are interested in consist of distributions, and are obtained through completion of $\mathcal{F}^{-1}C_c^\infty(\Omega)$ in the norm
	$$\|f\|_{B_{p,q}^{s}(\Omega)}:=\left(\sum_{j,i}a^{jsq}\|f\ast\phi_j^i\|_{L^p}^q\right)^{1/q}.$$ 
	We refer to \cite[Lemma 2.7]{bampouras2024besovspacesschattenclass} for the details, and familiar properties of duality and complex interpolation. We will use this subsection to record two simple additional estimates.
	\begin{lm}\label{estim}
		Let $\mathscr{P}(A_j^i)$ be a valid Fourier-partition. Then there exists $C>0$, independent of $j,i$, such that $$\|\phi_j^i\|_{L^p}\leq C a^{j(1-\frac{1}{p})}, \qquad p \geq 1, \; \phi_j^i \in \mathscr{P}(A_j^i).$$ 
	\end{lm}
	\begin{proof}
		Since $\{\phi_j^i\}\in\mathscr{P}(A_j^i)$, there is $C_1$ such that $\|\phi_j^i\|_{L^1}\leq C_1$. Now, let $\phi \in S(\mathbb R^n)$ be a Schwartz function satisfying:
		\begin{enumerate}
			\item $\widehat{\phi}(x)=1$ for $x\in (-1/2,1/2)^n$.
			\item $\sup_{1\leq p\leq \infty}\|\phi\|_{L^p}<\infty$.
		\end{enumerate}
		If $T_{j,i}$ is the affine bijection between $A_j^i$ and $(-1/2,1/2)^n$, by Young's inequality we have that 
		$$\|\phi_j^i\|_{L^p}=\|\phi_j^i\ast(\mathcal{F}^{-1}(\widehat{\phi}\circ T_{j,i}))\|_{L^p}\leq C_1\|\mathcal{F}^{-1}(\widehat{\phi}\circ T_{j,i})\|_{L^p}=C_1|\det T_{j,i}|^{\frac{1}{p}-1}\|\phi\|_{L^p},$$
		where $\mathcal{F}$ denotes the Fourier transform. The result follows from the fact that $|\det T_{j,i}|^{-1}= m(A_j^i)\approx a^j$.
	\end{proof}
	Using this lemma we can prove the expected Besov embedding theorem.
	\begin{theo}\label{embed}
		Let $1\leq p_1\leq p_2\leq \infty$, $1\leq q_1\leq q_2\leq \infty$. Then for every choice of real numbers $s_1,s_2$ with $s_2-s_1= \frac{1}{p_2}-\frac{1}{p_1}$, it is true that 
		$$B_{p_1,q_1}^{s_1}(\Omega)\hookrightarrow B_{p_2,q_2}^{s_2}(\Omega).$$
	\end{theo}
	\begin{proof}
		Let $f \in B^{s_1}_{p_1, q_1}(\Omega)$ be a Schwartz function. For a pair $(i,j)$, let $J_{j,i}=\{(k,l):A_k^l\cap A_j^i\neq \emptyset\}$. Let $s \geq 1$ be such that $1+\frac{1}{p_2}=\frac{1}{p_1}+\frac{1}{s}$. Since $\phi_j^i=\phi_j^i\ast(\sum_{(k,l)\in J_{j,i}}\phi_k^l)$,
		$$\|f\ast \phi_j^i\|_{L^{p_2}}\leq \sum_{(k,l)\in J_{j,i}}\|f\ast \phi_j^i\ast\phi_k^l\|_{L^{p_2}}\leq \|f\ast \phi_j^i\|_{L^{p_1}}\sum_{(k,l)\in J_{j,i}}\|\phi_k^l\|_{L^{s}}.$$
		By Lemma \ref{estim}, $\|\phi_k^l\|_{L^{s}}\lesssim a^{k(\frac{1}{p_1}-\frac{1}{p_2})}.$ By Axioms 1 and 5, we get that
		\[\|f\ast \phi_j^i\|_{L^{p_2}}\lesssim a^{j(\frac{1}{p_1}-\frac{1}{p_2})} \|f\ast \phi_j^i\|_{L^{p_1}}.\]
		Since $\ell^{q_1}\hookrightarrow \ell^{q_2}$, the result follows.
	\end{proof}
	
	\subsection{Density of test functions}
	In this short subsection, we establish that smooth functions with compactly supported Fourier transform in $\Omega$ are dense in $\PW^p(\Omega)$, for all $1 \leq p < \infty$.
	
	\begin{lm}\label{compactdense}
		Let $\Omega$ be an open and convex subset of $\mathbb{R}^n$. Then $\mathcal{F}^{-1}C_c^\infty(\Omega)$ is dense in $\PW^p(\Omega)$ for all $1\leq p< \infty$.
	\end{lm}
	\begin{proof}
		Let us fix $y\in \Omega$ and consider, for $0<r<1$, the contraction $\Omega_r=\overline{(1-r)y+r\Omega}$. For $f\in \PW^p(\Omega)$ we define 
		$$f_r(x)=r^n e^{2\pi i(1-r)\langle y, x \rangle}f(rx).$$
		Then 
		$$\hat{f}_r(x)=\hat{f}\left(\frac{x-(1-r)y}{r}\right),$$
		and thus $\supp\hat{f}_r\subset \Omega_r$.
		Moreover, $f_r\to f$ in $L^p$ as $r\to 1$, as follows by noting that $F \mapsto F_r$, $0 < r < 1$, is a uniformly bounded operation on all of $L^p$, and approximating $f$ with a Schwartz function $F$ in $L^p$. Convolving $f_r$ with an approximate identity and a mollifier provides smooth approximants with compact Fourier support in $\Omega$.
	\end{proof}
	\section{Boundary-weighted Fourier inequalities for general sets}\label{Sec3}
	In this section we will consider the $(p,q,d)$ boundary-weighted Fourier inequality \eqref{hardy},
	\[
	\int_{\Omega}\dfrac{|\hat{f}(x)|^p}{\omega_{\Omega}^d(x)}dx\leq C\|f\|_{L^q}^p,\quad f\in \PW^q(\Omega),
	\]
	for a general open convex set $\Omega\subset\mathbb{R}^n$ not containing any affine lines, where $1\leq p,q< \infty$ and $d\in\mathbb{R}$, obtaining, in particular, a sharp range of necessary conditions. 
	
	To start, note that for $q > 2$, the space $\mathcal{F}\PW^q(\Omega)$ contains distributions which may not be functions. Accordingly, inequality \eqref{hardy} requires that $q \leq 2$.
	\begin{lm}\label{q>2}
		The boundary-weighted Fourier inequality \eqref{hardy} fails whenever $q>2$.
	\end{lm}
	\begin{proof}
		By choosing $B$ a ball with $\overline{B}\subset \Omega$, inequality \eqref{hardy} implies the inequality
		$$\|\hat{f}\|_{L^p}\lesssim \|f\|_{L^q}, \quad f\in \PW^q(B).$$ 
		This inequality is of course known not to hold, but let us provide a simple proof of this fact.
		
		Suppose for contradiction that \eqref{hardy} holds when $\Omega$ is the unit ball. Choose $\phi \in S(\mathbb R^n)$ such that $\hat{\phi}(x)=1$ on $B(0,\frac{1}{2})$ and $\supp\hat{\phi}= \overline{B}(0,1)$. Given $\delta>0$ we can find $N=N(\delta) \approx \delta^{-n}$ points $y_j$ such that the closed balls $\overline{B}(y_j,\delta)\subset B(0,1)$ are disjoint. Let $\hat{\phi}_j(x)=\hat{\phi}(\frac{x-y_j}{\delta})$. For $\epsilon=(\epsilon_1,...,\epsilon_N)\in \{-1,1\}^N$ let us define 
		$$	\hat f_\epsilon(x)= \sum_{j=1}^N \epsilon_j\hat{\phi}_j(x).$$ 
		Since $\supp\hat{\phi}_j= \overline{B}(y_j,\delta)$ are mutually disjoint, we have that
		$$\|\hat f_\epsilon\|_{L^p}^p = \sum_{j=1}^N \|\hat{\phi}_j\|_{L^p}^p  \approx N \delta^n \approx 1.$$
		On the other hand, applying Khintchine's inequality \cite{HaagerupKhintchine} pointwise to $f_\epsilon = \sum \epsilon_j \phi_j$ and integrating gives that 
		$$2^{-N}\sum_{\epsilon\in \{-1,1\}^N}\|f_\epsilon\|_{L^q}^q\approx \big\|\left(\sum_{j=1}^N |\phi_j|^2\right)^{\frac{1}{2}}\big\|_{L^q}^q.$$
		By the contradiction hypothesis we therefore have that
		$$ \big\|\left(\sum_{j=1}^N |\phi_j|^2\right)^{\frac{1}{2}}\big\|_{L^q}^q\gtrsim 2^{-N}\sum_{\epsilon\in \{-1,1\}^N}\|\hat f_\epsilon\|_{L^p}^q\approx 1.$$
		But, by Fourier inversion, $\phi_j(x)=e^{2\pi i x\cdot y_j} \delta^n \phi(\delta x),$ and thus
		$$ \big\|\left(\sum_{j=1}^N |\phi_j|^2\right)^{\frac{1}{2}}\big\|_{L^q}^q=\delta^{qn}N^{\frac{q}{2}}\big\|\phi(\delta  \, \cdot )\big\|_{L^q}^q\approx \delta^{nq-n}N^{\frac{q}{2}} \approx \delta^{n(\frac{q}{2} - 1)}.$$ 
		We reach a contradiction taking $\delta \to 0^+$, since $q > 2$.
	\end{proof}
	The setting of $q=2$ is straightforward, described in terms of the set $A(\Omega)$ of real numbers $r$ such that $\omega^{-r}_\Omega\in L^1(\Omega)$. When $\Omega$ is a polytope we have that $A(\Omega)=(-\infty, 1)$, and when $\Omega$ is a set with positive affine surface area we have that $A(\Omega)=(-\infty,\frac{2}{n+1})$, see \cite{MR1194970}. Domains with zero affine surface area (e.g. cylinders) need to be treated on a case-by-case basis.
	
	\begin{lm} \label{lem:qis2}
		Let $q=2$, $p< 2$. Then the boundary-weighted Fourier inequality \eqref{hardy} holds 
		\begin{enumerate}
			\item if and only if $d\in \frac{2-p}{2} A(\Omega)$, when $\Omega$ is bounded;
			\item never, when $\Omega$ is unbounded.
		\end{enumerate}
		For $p=q=2$, the boundary-weighted Fourier inequality holds if and only if $d\leq 0$ for $\Omega$ bounded and if and only if $d=0$ for $\Omega$ unbounded.
	\end{lm}
	\begin{proof}
		Let $p<2$. The boundary-weighted Fourier inequality for $q=2$ becomes the embedding $L^2(\Omega)\hookrightarrow L^p(\Omega,\omega_\Omega^{-d})$ and thus this holds if and only if $\omega_\Omega^{-d}\in L^{\frac{2}{2-p}}(\Omega)$. This gives (1).
		
		When $\Omega$ is unbounded, choose, for each sufficiently separated index $j$ a parallelepiped $A_j=A_j^{i_j}$ from an $a$-admissibility family $\{A_{j}^i\}$ for $\Omega$, so that $A_j$ are mutually disjoint and $m(A_j) \approx a^j$. Then
		$$\int_{\Omega}\omega^{r}_\Omega(x) \, dx\gtrsim \sum_{j\in \mathbb{Z}}\int_{\Omega}\omega^{r}_\Omega(x)\chi_{A_j^{i_j}} \, dx \approx \sum_{j\in \mathbb{Z}}a^ja^{jr},$$ which is never finite.
		
		The case $p=2$ becomes $\omega^{-d}_\Omega\in L^{\infty}(\Omega)$ which happens only when $d\leq0$ for $\Omega$ bounded and when $d=0$ if $\Omega$ is unbounded.
	\end{proof}
	Therefore the interesting case is $1\leq q< 2$. In this case, there is an immediate candidate $d_c=d_c(p,q)$ for the best possible exponent.
	\begin{lm}\label{bestd}
		Let 
		$$d_c = d_c(p,q) = 1+\frac{p}{q}-p.$$
		For the boundary-weighted Fourier inequality \eqref{hardy} to hold it is necessary that
		$$d \leq d_c, \textrm{ and } p \leq q',$$
		where $\frac{1}{q}+\frac{1}{q'}=1$. Furthermore, if $\Omega$ is unbounded then it is necessary that $d = d_c$.
	\end{lm}
	\begin{proof}
		Since $\Omega$ is admissible by Theorem \ref{admissibility}, we may consider a valid Fourier-partition $\mathscr{P}(A_j^i) = \{\phi_j^i\}$ associated with an admissibility cover $\{A_j^i\}$. We may additionally assume that $\widehat{\phi}_j^i\geq C > 0$ on a fixed contraction of $A_j^i$ (independent of $(i,j)$), see \cite[Proposition 2.6]{bampouras2024besovspacesschattenclass}. Thus    $$\int_{\Omega}\dfrac{|\widehat{\phi}_j^i(x)|^p}{\omega_{\Omega}^d(x)}dx\approx a^{j(1-d)}.$$
		By Lemma \ref{estim}, it also holds that 
		$$\|\phi_j^i\|_{L^q}^p\lesssim a^{pj(1-\frac{1}{q})}.$$
		Therefore inequality \eqref{hardy} implies that $d \leq d_c$, and that $d = d_c$ if $\Omega$ is unbounded.
		
		Now for the case $p>q'$ observe that for a ball $B$, as in the proof of Lemma~\ref{q>2}, inequality \eqref{hardy} implies the inequality 	$$\|\hat{f}\|_{L^p}\lesssim \|f\|_{L^q}, \quad f\in \PW^q(B),$$
		for a ball $B$, $\overline{B}\subset \Omega$. We have just seen that this requires $0 \leq d_c,$ that is, $p \leq q'$.
	\end{proof}
	One important example with "critical exponent" $d = d_c$ is furnished by $(p,q,d)=(2,1,1)$,
	\begin{equation*}
		\int_\Omega \dfrac{|\hat{f}(x)|^2}{\omega_{\Omega}(x)}dx\leq C\|f\|_{L^1}^2, \quad f\in\PW^{1}(\Omega).
	\end{equation*}
	This is Helson's inequality, discussed in the introduction. The following lemma gives a sufficient criterion for the validity of Helson's inequality, to be used later for the unit ball. 
	
	\begin{lm}\label{1/omega}
		Suppose that there exists $\psi \in L^\infty$ such that the restriction $\hat{\psi}|_{\Omega}$ (as a distribution) coincides with $\omega_\Omega^{-1}$ in $\Omega$. 
		Then the $(2,1,1)$ boundary-weighted Fourier inequality \eqref{eq:helsonineq} holds.
	\end{lm}
	\begin{proof}
		By Lemma~\ref{compactdense}, it is enough to consider $f \in \mathcal{F}^{-1}C_c^\infty(\Omega)$. Then
		\begin{eqnarray} 
			\int_\Omega \dfrac{|\hat f(x)|^2}{\omega_{\Omega}(x)}dx&=&\int_\Omega |\hat f(x)|^2\hat{\psi}(x) \, dx=\langle \hat f \hat\psi ,\hat f\rangle = \langle f \ast \psi, f \rangle\nonumber \\   &\leq& \|f \ast \psi\|_{L^{\infty}}\|f\|_{L^{1}}\leq \|\psi\|_{L^{\infty}}\|f\|_{1}^2. \nonumber 
		\end{eqnarray} 
		This is the desired inequality.
	\end{proof}
	
	\begin{remark}
		By the same argument, if $\omega_\Omega^{-d}$ coincides in $\Omega$ with the Fourier transform of a function in the weak-type space $L^{p,\infty}$, $1 < p < \infty$, then the $(2, 2p/(2p-1), d)$-inequality holds.
	\end{remark}
	
	When $1 \leq p < q$, the exponent $d_c(p,q) = 1 + p/q - p$ is not the right one, unless $\Omega$ belongs to a certain subclass of domains with zero affine surface area (which includes polytopes, see Section~\ref{Sec4}).
	
	\begin{theo}\label{plessqd=dc}
		Suppose that $1\leq p<q<2$.
		If the $(p,q,d)$ boundary-weighted Fourier inequality
		\[
		\int_\Omega \frac{|\widehat f(x)|^p}{\omega_\Omega(x)^d}\,dx
		\lesssim \|f\|_{L^q}^p,
		\qquad f\in\PW^q(\Omega),
		\]
		holds, then $\Omega$ is bounded and $d<d_c$.
		
		Moreover, if $n\geq2$ and $\Omega$ has positive affine surface area, then it is necessary that
		\[
		d\leq d_c-\frac{(n-1)(q-p)}{(n+1)q}.
		\]
	\end{theo}
	
	\begin{proof}
		
		Let $\{A_j^i\}$ be a cover with Fourier-partition $\{\phi^i_j\}$. Let $x_j^i$ denote the center of $A_j^i$. By the proof of Theorem~\ref{admissibility}, we may assume that there is a fixed $\rho\in(0,1)$
		such that the contractions
		\[
		A_j^i(\rho)=(1-\rho)x_j^i+\rho A_j^i,
		\qquad i\in I_j,
		\]
		are mutually disjoint. Define
		\[
		\psi_j^i(x)
		=
		\rho^n e^{2\pi i(1-\rho)x\cdot x_j^i}\phi_j^i(\rho x), \qquad \hat{\psi}_j^i(x)
		=
		\hat{\phi}_j^i \left(
		\frac{x-(1-\rho)x_j^i}{\rho}
		\right).
		\]
		Then $\supp\widehat{\psi_j^i}\subset A_j^i(\rho)$ and
		\begin{equation*}
			\|\psi_j^i\|_{L^q}^q\lesssim a^{j(q-1)},
			\qquad
			\int_\Omega
			\frac{|\widehat{\psi_j^i}(x)|^p}{\omega_\Omega(x)^d}\,dx
			\approx a^{j(1-d)},
		\end{equation*}
		by Lemma~\ref{estim}. Let $\mathcal E$ be any finite set of indices $(i,j)$, $i\in I_j$, chosen so that the family $\{A_j^i(\rho)\}_{(i,j)\in \mathcal{E}}$ is pairwise disjoint; for instance, it suffices that the distinct $j-$levels occurring in $\mathcal E$ are separated by more than $2M$, where $M$ is the constant of \ref{Axiom1}. For scalars $c_{i, j}$ and signs
		$\epsilon = (\epsilon_{i,j})_{i,j \in \mathcal{E}} \in\{-1,1\}^{\mathcal{E}}$, let
		\[
		g_\epsilon
		=
		\sum_{(i, j)\in\mathcal E}
		\epsilon_{i, j}c_{i,j}\psi_j^i.
		\]
		Assume that the $(p,q,d)$ boundary-weighted Fourier inequality holds. The left-hand side of the
		$(p,q,d)$-inequality, applied to $f = g_\epsilon$, is independent of the choice of signs. Taking the
		power $q/p$, averaging over the signs, and then applying Khintchine's
		inequality pointwise, we obtain
		\begin{align*}
			\left(
			\sum_{(i, j)\in\mathcal E}
			|c_{i,j}|^p a^{j(1-d)}
			\right)^{q/p}
			&\lesssim
			2^{-|\mathcal E|} \sum_{\epsilon \in \{-1,1\}^{\mathcal{E}}} \|g_\epsilon\|_{L^q}^q 
			\approx
			\int_{\mathbb R^n}
			\left(
			\sum_{(i,j)\in\mathcal E}
			|c_{i,j}\psi_j^i(x)|^2
			\right)^{q/2}dx \\
			&\leq
			\sum_{(i, j)\in\mathcal E}
			|c_{i, j}|^q\|\psi_j^i\|_{L^q}^q \notag
			\lesssim
			\sum_{(i, j)\in\mathcal E}
			|c_{i, j}|^q a^{j(q-1)}.
			\label{eq:packet-obstruction}
		\end{align*}
		Here we used that $q/2<1$ in the second inequality.
		
		Renormalizing the coefficients, this is
		\begin{equation}\label{eq:normalized-obs}
			\left(
			\sum_{(j,i)\in\mathcal E}
			|b_{j,i}|^p a^{j(d_c-d)}
			\right)^{q/p}
			\lesssim
			\sum_{(j,i)\in\mathcal E}|b_{j,i}|^q.
		\end{equation}
		When $d = d_c$, this gives a contradiction as the cardinality of $\mathcal E$ tends to infinity, since $p < q$. By Lemma~\ref{bestd}, we conclude that any valid inequality must have $d<d_c$, and that $\Omega$ must be bounded.
		
		It remains to prove the sharper assertion when
		$\Omega$ has positive affine surface area. 
		Let $I_j$ be the set of all $i$s such that $A_j^i$ is defined. By Schmuckenschl\"{a}ger \cite{MR1194970} (see also \cite[Theorem 2.3]{bampouras2024besovspacesschattenclass}) we then know that
		$a^{\frac{2j}{n+1}}\approx m(\Delta_j(\Omega))$. By \ref{Axiom5}, the sets $A_j^i$ have uniformly bounded
		overlap. Moreover, by the proof of \ref{Axiom3}, a uniform dilation of them covers $\Delta_j$, while \ref{Axiom1} gives
		$$A_j^i\subset\omega_\Omega^{-1}[a^{j-M},a^{j+M}].$$ Thus
		$$	a^{\frac{2j}{n+1}}\approx m(\Delta_j(\Omega))\approx \sum_{i\in I_j}m( A_j^i)\approx \# I_j a^j,$$ and therefore
		$$\# I_j\approx a^{\frac{(1-n)j}{n+1}}.$$
		For a fixed sufficiently negative $j$, apply
		\eqref{eq:normalized-obs} with
		$\mathcal E_j=\{(i,j):i\in I_j\}$ and $b_{i, j}=1$. We obtain
		\[
		(\# I_j)^{1-p/q}a^{j(d_c-d)}\lesssim1.
		\]
		Letting $j\to-\infty$ yields
		\[
		d_c-d
		\geq
		\frac{n-1}{n+1}\left(1-\frac{p}{q}\right)
		=
		\frac{(n-1)(q-p)}{(n+1)q},
		\]
		which is the claimed bound.
	\end{proof}

	\section{Polytopes and polyhedra}\label{Sec4}
	In this section we are going to completely characterize the valid $(p,q,d)$ boundary-weighted Fourier inequalities for polyhedra. To do so, we are first going to generalize the fact that the analytic Hardy spaces $H^p \simeq \PW^p(\R_+^n)$ sit on an interpolation scale for $1 \leq p < \infty$, and thus prove that we may indeed interpolate between $\PW^1(P)$ and $\PW^p(P)$, for polyhedra $P$. We do not know if it is possible to interpolate between $\PW^{p_1}(\Omega)$ and $\PW^{p_2}(\Omega)$ for any other types of domains $\Omega$, not even for $p_1, p_2 > 1$.
	
	\subsection{Complex interpolation for polyhedra} 
	Given an open polyhedron $P\subset\mathbb{R}^n$ which does not contain any lines, let $\Pi_P:L^2(\mathbb{R}^n)\to \PW^2(P)$ be the orthogonal projection, that is, the Fourier multiplier 
	$$\hat{\Pi}_P f=\hat{f}\chi_P, \qquad f \in L^2.$$ We will first observe that $\Pi_P$ extends boundedly to $L^p$, $1<p<\infty$, as a consequence of the boundedness of the Hilbert transform on $L^p(\R)$, or, equivalently, of the Riesz projection $\Pi_{\R_+} \colon L^p(\R) \to \PW^p(\R_+)$. This is in stark contrast to the case with curvature; when $\Omega = B$ is a ball in $\R^n$, $n\geq 2$, it is a famous result \cite{MR296602} that the corresponding ball Fourier multiplier never defines a bounded operator $L^p(\R^n) \to \PW^p(B)$, unless $n =1 $ or $p = 2$.
	\begin{lm}\label{projbound}
		Let $1 < p < \infty$. The projection $\Pi_P:L^p(\mathbb{R}^n)\to \PW^p(P)$ is bounded. Furthermore, if $P$ has $N$ facets, then 
		$$\|\Pi_P f\|_{L^p}\leq C_p^N\|f\|_{L^p},$$
		where $C_p$ is the norm of $\Pi_{\R_+} \colon L^p(\R) \to \PW^p(\R_+)$.
	\end{lm}
	\begin{proof}
		For a half space $H=\{(x_1,...,x_n):x_1>0\}$, it is evident that
		$$\|\Pi_H\|_{L^p(\mathbb{R}^n)\to \PW^p(H)} = C_p.$$ Let $P$ be a polyhedron with $N$ facets. Then $P$ is a finite intersection of half-spaces $H_j$, $j=1,...,N$, and thus $\Pi_P=\Pi_{H_1} \Pi_{H_2} \cdots \Pi_{H_N}$. 
	\end{proof} 
	
	This existence of a bounded projection immediately implies that the spaces $\PW^p(P)$ sit on a complex interpolation scale for $1 < p < \infty$. However, the Fourier multiplier given by $\chi_P$ is never bounded on $L^1$, so we must instead exploit the classical interpolation theorem for product Hardy spaces. 
	
	We begin with a geometric lemma. 
	
	\begin{lm}\label{polyhedronpartition}
		Up to a set of measure zero, $P$ can be written as a finite disjoint union of sets affinely equivalent to $\Sigma_k\times\mathbb R_+^\ell$, where $k+\ell=n$, and $\Sigma_k$ is the simplex
		\[
		\Sigma_k=
		\{x\in\mathbb R_+^k:\sum_{r=1}^k x_r<1 \}.
		\]
	\end{lm}
	
	\begin{proof}
		By the Minkowski--Weyl theorem \cite[Theorem 14.3]{MR2335496}, we can write $P$ as the sum of a convex hull of finitely many points and the conic hull of a finite number of rays,
		\[
		\overline P
		=
		\conv\{u_1,\ldots,u_N\}
		+
		\cone\{v_1,\ldots,v_M\}.
		\]
		Introduce one additional coordinate and consider the polyhedral cone
		\[
		K
		=
		\cone\bigl(
		\{(u_r,1):1\leq r\leq N\}
		\cup
		\{(v_s,0):1\leq s\leq M\}
		\bigr)
		\subset\mathbb R^{n+1},
		\]
		so that
		\[
		K\cap\bigl(\mathbb R^n\times\{1\}\bigr)
		=
		\overline P\times\{1\}.
		\]
		By the same hypothesis for $P$, the cone $K$ cannot contain any lines (it is \textit{pointed}). We can therefore triangulate $K$ into a finite decomposition
		\[
		K=\bigcup_\alpha K_\alpha
		\]
		of full-dimensional simplicial cones whose interiors are pairwise disjoint, see for example \cite[Proposition~2.2.4]{DeLoeraRambauSantos2010}.
		
		Fix one such cone $K_\alpha$. Renormalizing the generators with positive $(n+1)$:th coordinate, and reordering, we may write
		\[
		K_\alpha
		=
		\cone\bigl\{
		(y_0,1),\ldots,(y_k,1),
		(w_1,0),\ldots,(w_\ell,0)
		\bigr\},
		\qquad k+\ell=n.
		\]
		At height one, we must then have that
		\[
		K_\alpha' := K_\alpha\cap
		\bigl(\mathbb R^n\times\{1\}\bigr)
		=
		\left(
		\conv\{y_0,\ldots,y_k\}
		+
		\cone\{w_1,\ldots,w_\ell\}
		\right)\times\{1\},
		\]
		that is, the interior of $K_\alpha'$ is affinely equivalent to $\Sigma_k\times\mathbb R_+^\ell$.
		The cones $K_\alpha$ cover $K$ and meet only along their
		faces. The sets $K_\alpha'$ therefore cover $P$ and overlap only along their boundaries, which have measure zero.
	\end{proof}
	
	We now prove the theorem of this subsection.
	
	\begin{theo}\label{interpolation}
		Let $1<p<\infty$ and $0<\theta<1$. Then we have the complex interpolation identity
		$$(\PW^1(P),\PW^p(P))_\theta=\PW^q(P), \qquad \frac{1}{q}=(1-\theta)+\frac{\theta}{p}.$$
	\end{theo}
	
	\begin{proof}
		The non-immediate inclusion to be proved is that $$\PW^q(P) \hookrightarrow (\PW^1(P),\PW^p(P))_\theta.$$
		The converse inclusion is always true since $$(\PW^1(P),\PW^p(P))_\theta\hookrightarrow (L^1(\mathbb{R}^n),L^p(\mathbb{R}^n))_\theta=L^q(\mathbb{R}^n),$$ and $\PW^1(P)+\PW^p(P)$ has Fourier support in $\overline{P}$.
		
		We next record the orthant case, equivalent to interpolation for product Hardy spaces. 
		Let $H_{\mathrm{prod}}^r$, $1 \leq r < \infty$, denote the real product
		Hardy space.  The classical product Hardy space interpolation theorem, see for example
		\cite[Theorem~4 and Application~(2)]{CwikelMilmanSagher1986}, gives
		\[
		\bigl(H^1_{\mathrm{prod}},H^p_{\mathrm{prod}}\bigr)_\theta
		=H^q_{\mathrm{prod}}=L^q.
		\]
		The product Riesz projection $\Pi_{\R_+^n}$
		is bounded both on $H^1_{\mathrm{prod}}$ and on
		$L^p$, mapping $H^1_{\mathrm{prod}}$ onto the analytic Hardy space $\PW^1(\R_+^n)$ and $L^p$ onto $\PW^p(\R_+^n)$. Applying the
		retract theorem, and then an affine change of variables, proves that
		\begin{equation}\label{eq:orthant-interpolation}
			\bigl(\PW^1(O),\PW^p(O)\bigr)_\theta=\PW^q(O)
		\end{equation}
		for every orthant $O$, affinely equivalent to $\R_+^n$.
		
		We now consider the model polyhedron of Lemma~\ref{polyhedronpartition}
		\[
		T=\Sigma_k\times\mathbb R_+^\ell,
		\qquad k+\ell=n.
		\]
		We just dealt with the case $k=0$, so suppose that $k\geq1$. For each face of $\Sigma_k$, let $C_\nu$, $1 \leq \nu \leq k+1$, be the region obtained by discarding the inequality defining that face, and let $O_\nu=C_\nu\times\mathbb R_+^\ell$ denote the resulting affine orthant.  Choose also
		$\alpha_\nu,\beta_\nu\in C_c^\infty(\mathbb R^k)$ so that $\{\alpha_\nu\}$ is a partition of unity in a neighborhood of $\overline{\Sigma_k}$ such that each $\alpha_\nu$ is supported away from the corresponding face, and then choose $\beta_\nu$ equal to one in a neighborhood of $\supp\alpha_\nu$, also supported away from the same face.
		
		For $a\in C_c^\infty(\mathbb R^k)$, let $M_a$ be the partial Fourier
		multiplier
		\[
		\widehat{M_a f}(\xi',\xi'')
		=
		a(\xi')\widehat f(\xi',\xi''),
		\qquad
		(\xi',\xi'')\in\mathbb R^k\times\mathbb R^\ell.
		\]
		Then $M_a \colon L^r \to L^r$ is bounded for every $r$,
		\begin{equation*}
			\|M_a\|_{L^r\to L^r}
			\leq\|\mathcal{F}^{-1} a\|_{L^1(\mathbb R^k)},
			\qquad 1\leq r\leq\infty.
		\end{equation*}
		For $f\in\PW^q(T)$, let $f_\nu=M_{\alpha_\nu}f$, giving us the decomposition  
		$$f=\sum_\nu f_\nu, \qquad f_\nu\in\PW^q(O_\nu).$$
		Furthermore, $M_{\beta_\nu}f_\nu=f_\nu$, and 
		$M_{\beta_\nu}$ maps $\PW^r(O_\nu)$ boundedly into $\PW^r(T)$, in particular for $r = 1$ and $r = p$.  The interpolation identity \eqref{eq:orthant-interpolation} and the functorial property of interpolation thus gives
		\[
		\|f_\nu\|_{(\PW^1(T),\PW^p(T))_\theta}
		\lesssim
		\|f_\nu\|_{(\PW^1(O_\nu),\PW^p(O_\nu))_\theta}
		\approx
		\|f_\nu\|_{\PW^q(O_\nu)}
		\lesssim
		\|f\|_{L^q}.
		\]
		Summing over the finitely many $\nu$ proves that
		\begin{equation} \label{eq:modelint}
			\PW^q(T)
			\hookrightarrow
			\bigl(\PW^1(T),\PW^p(T)\bigr)_\theta.
		\end{equation}
		The same conclusion holds for every affine image of $T$, by a change of variable.
		
		To conclude, use Lemma~\ref{polyhedronpartition} to write
		$P=\bigcup_{m \leq N} T_m$ up to null sets, where every $T_m$ is an
		affine image of a model polyhedron. Lemma
		\ref{projbound} shows that the Fourier projections $\Pi_{T_m}$ are bounded
		on $L^q$, and that $ \sum_{m\leq N}\Pi_{T_m}=\Pi_P.$ By the triangle inequality and \eqref{eq:modelint}, we therefore find that
		\begin{align*}
			\|f\|_{(\PW^1(P),\PW^p(P))_\theta} 
			&\leq
			\sum_{m=1}^N
			\| \Pi_{T_m}f\|_{(\PW^1(T_m),\PW^p(T_m))_\theta} \\
			&\lesssim
			\sum_{m=1}^N\|\Pi_{T_m} f\|_{\PW^q(T_m)}
			\lesssim
			\|f\|_{L^q},
		\end{align*}
		as desired.
	\end{proof}

	\subsection{Boundary-weighted Fourier inequalities for polyhedra}
	Before we prove our main theorem for this section, we will first need to extend the Hardy inequality of \cite[Theorem 1.2]{bampouras2026nehari} to all polyhedra.
	\begin{lm}\label{polyhbamp}
		Suppose that $P$ has $N$ facets. There exists $C=C(n,N)$ such that the inequality
		$$\int_P \dfrac{|\hat{f}(x)|}{\omega_P(x)}dx\leq C\|f\|_{L^1},\quad f\in \PW^1(P),$$ holds.
	\end{lm} 
	\begin{proof}
		This result is shown in \cite{bampouras2026nehari} for polytopes, that is, for bounded polyhedra. To generalize it to unbounded polyhedra, note, 
		by Lemma \ref{compactdense}, that it suffices to treat functions $f \in \mathcal{F}^{-1}C_c^\infty(P)$. First, let us notice that $P$ can be written as the union of (bounded) polytopes $P_k$ with a uniformly bounded number of facets, satisfying $P_k\subset P_{k+1}$. To see this, let $P_k=P\cap (-k,k)^n$ and notice that the number of facets of $P_k$ is bounded by the number of facets of $P$ plus the $2n$ facets of $(-k,k)^n$. Since $f\in \mathcal{F}^{-1}C_c^\infty(P)$, we can find $k$ such that $\supp\hat{f}\subset P_k$. By \cite[Theorem 1.2]{bampouras2026nehari}, there is a constant $C(n,N+2n)$ such that
		$$\int_{P_k} \dfrac{|\hat{f}(x)|}{\omega_{P_k}(x)}dx\leq C(n,N+2n)\|f\|_{L^1},$$
		and thus, since $\omega_{P_k}\leq \omega_{P},$
		\[\int_{P} \dfrac{|\hat{f}(x)|}{\omega_{P}(x)}dx\leq \int_{P_k} \dfrac{|\hat{f}(x)|}{\omega_{P_k}(x)}dx\leq C(n,N+2n)\|f\|_{L^1}. \qedhere \]
	\end{proof}

	By interpolation, we can now establish that when  $q \leq p$, the $(p, q, d)$ Fourier inequality is always true for the critical exponent $d = d_c(p,q) = 1+\frac{p}{q}-p$, as long as the necessary conditions $q \leq 2$ and $p\leq q'$ hold.
	\begin{theo}\label{hardyall}
		For $1\leq q<2$ and $q\leq p \leq q'$, the boundary-weighted Fourier inequality
		$$\int_{P} \dfrac{|\hat{f}(x)|^p}{\omega_P^{1+\frac{p}{q}-p}(x)}dx\leq C\|f\|_{L^q}^p, \quad f\in\PW^q(P),$$
		holds.
	\end{theo}
	
	\begin{proof}
		Since the boundary-weighted Fourier inequality holds for $p=q=d=1$, by Lemma \ref{polyhbamp}, the Fourier transform is a bounded operator
		$$\mathcal{F}:\PW^1(P)\to L^1(P,\omega_P^{-1}).$$ 
		For $q = 1$ the result is then immediate:
		$$\int_P\dfrac{|\hat{f}(x)|^p}{\omega_P(x)}dx\leq \|f\|_{L^1}^{p-1}\int_P\dfrac{|\hat{f}(x)|}{\omega_P(x)}dx.$$
		When $1 < q < 2$, we continue by noting that
		$$\mathcal{F}:\PW^2(P)\to L^2(P).$$ 
		Interpolating with Theorem \ref{interpolation}, we derive that $\mathcal{F}:\PW^q(P)\to L^q(P,\omega_P^{q-2})$ is bounded, that is,
		$$\int_P\dfrac{|\hat {f}(x)|^q}{\omega_P^{2-q}(x)}dx\leq C\|f\|_{L^q}^q.$$
		Let $\theta\in [0,1]$ be such that $p=\theta q +(1-\theta)q'$, so that 
		$$1+\frac{p}{q}-p=\theta(2-q).$$ 
		Then, by H\"older's inequality, we get that
		\begin{eqnarray} \int_P \dfrac{|\hat{f}(x)|^p}{\omega_P^{1+\frac{p}{q}-p}(x)}dx&=&\int_P|\hat{f}(x)|^{(1-\theta)q'} \dfrac{|\hat{f}(x)|^{\theta q}}{\omega_P^{\theta(2-q)}(x)}dx \nonumber \\
			&\leq& 
			\left(\int_P |\hat{f}(x)|^{q'} dx\right)^{1-\theta} \left(\int_P \dfrac{|\hat{f}(x)|^{ q}}{\omega_P^{(2-q)}(x)}dx\right)^\theta \nonumber \\ 
			&\lesssim& 
			\|f\|_{L^q}^{q'(1-\theta)}\|f\|_{L^q}^{q\theta}=\|f\|_{L^q}^p, \nonumber 
		\end{eqnarray} 
		as desired.
	\end{proof}
	When $1\leq p<q<2$ we can also give a complete answer.
	
	\begin{prop}
		Let $1\leq p<q<2$. Then the following hold.
		\begin{enumerate}
			\item If $P$ is bounded, the boundary-weighted Fourier inequality 
			$$\int_P \dfrac{|\hat{f}(x)|^p}{\omega_P^d(x)}dx\lesssim \|f\|_{L^q}^p, \quad f\in \PW^q(P),$$ holds if and only if $d< d_c(p,q) = 1+\frac{p}{q}-p$.
			\item If $P$ is unbounded, the boundary-weighted Fourier inequality $$\int_P \dfrac{|\hat{f}(x)|^p}{\omega_P^d(x)}dx\lesssim \|f\|_{L^q}^p, \quad f\in \PW^q(P),$$ fails for every $d\in\mathbb{R}$.
		\end{enumerate}
	\end{prop}
	\begin{proof}
		By Lemma~\ref{bestd} and Theorem~\ref{plessqd=dc}, we have already covered the case of unbounded $P$, and shown that it is necessary that $d < 1+\frac{p}{q}-p$ when $P$ is bounded. 
		
		Therefore let us assume that $P$ is a polytope and $d<1+\frac{p}{q}-p$. By \cite{Schuett1991}, we know that $\omega_P^\beta\in L^1(P)$ if and only if $\beta>-1$. Therefore, by H\"{o}lder's inequality and Theorem \ref{hardyall}, we get that
		\begin{eqnarray}
			\int_P \dfrac{|\hat{f}(x)|^p}{\omega_P^d(x)}dx&=&\int_P \left(\dfrac{|\hat{f}(x)|^q}{\omega_P^{2-q}(x)}\right)^{p/q}\omega_P^{p(2-q)/q-d}(x) \, dx \nonumber \\
			&\leq& \left(\int_P \dfrac{|\hat{f}(x)|^q}{\omega_P^{2-q}(x)}dx\right)^{p/q}\left(\int_P \omega_P^{(p(2-q)/q-d)q/(q-p)}(x)dx\right)^{(q-p)/q}\nonumber \\
			&\lesssim& \|f\|_{L^q}^p, \nonumber
		\end{eqnarray}
		where the last inequality holds since $$(p(2-q)/q-d)q/(q-p)>-1.$$ This completes the proof.
	\end{proof}

	\section{Boundary-weighted Fourier inequalities for the ball}\label{Sec5}
	In this and the next section we will explore boundary-weighted Fourier inequalities for the unit ball $B \subset \R^n$. We start by verifying the sufficient criterion established in Lemma~\ref{1/omega}, implying the validity of Helson's inequality. We do so through a nearly explicit formula. 
	\begin{prop}\label{prop:symbol-ball}
		There exists $\psi\in L^\infty(\mathbb{R}^n)$ such that 
		$$\hat{\psi}|_{B}=\frac{1}{\omega_{B}},$$
		as distributions in $B$. 
	\end{prop}
	
	\begin{proof}
		Let $\chi$ be a smooth function supported in a sufficiently thin annulus $R = \{1-\epsilon<|x|<1+\epsilon\}$, with $\chi \equiv 1$ in $\{1-\epsilon/2<|x|<1+\epsilon/2\}$. Let 
		$$P(x) = 1 - |x|^2, \qquad \alpha = (n+1)/2.$$ On $B$, a computation shows that 
		$$\omega_B^{-1}= aP^{-\alpha} = (1-\chi)\omega_B^{-1} + \chi a P^{-\alpha},$$
		where $a$ is real analytic and non-vanishing in the annulus $R$. Noting that $(1-\chi)\omega_B^{-1} \in L^1 \subset \mathcal{F} L^\infty$ and that $\chi a \in C^\infty_c(\R^n) \subset \mathcal{F} L^1$, it suffices to find a distribution $T_\alpha \in \mathcal{F} L^\infty$ which coincides with $P^{-\alpha}$ on the open unit ball $B$. 
		
		For this purpose, let 
		\[
		u_\alpha
		=
		\frac12\left((t+i0)^{-\alpha}+(t-i0)^{-\alpha}\right) 
		\in \mathcal S'(\mathbb R),
		\]
		where the tempered distributions $(t \pm i0)^{-\alpha}$ are obtained by analytic continuation of the locally integrable distributions
		\begin{equation} \label{eq:distdef}
			(t \pm i0)^{\lambda} = \lim_{\varepsilon \to 0^+} (t \pm i\varepsilon)^\lambda, \qquad \Real \lambda > -1.
		\end{equation}
		Actually the distributional equality \eqref{eq:distdef} remains valid for non-integer $\lambda = -\alpha$ (corresponding to the even dimensions $n$), while for $\lambda = -1, -2, \ldots,$ $(t \pm i0)^{\lambda}$ can be defined by the usual jump formula $(t \pm i0)^{-1} = \operatorname{pv} t^{-1} \mp i\pi \delta_0$, and differentiation thereof. See \cite[Sections~3.6 and 4.4] {GelfandShilov1964} for details. Furthermore, since $P$ is a submersion in a neighbourhood of $\partial B$, we may consider the pullback distribution
		$$T_\alpha = P^\ast u_\alpha = 	\frac12\left(
		(1-|x|^2+i0)^{-\alpha}
		+
		(1-|x|^2-i0)^{-\alpha}
		\right).$$
		As distributions inside $B$ we of course have that $T_\alpha|_B = P^{-\alpha}$.
		
		We now use a classical Bessel kernel
		formula, see \cite[Ch. III, Sec. 2.8, Eqs. (10) and (11)]{GelfandShilov1964}. It says, adapted to our Fourier convention, that for $\Real \lambda > n/2$, we have, as a distributional identity,
		\[
		\mathcal F^{-1}
		\left((1-|x|^2\pm i0)^{-\lambda}\right)(y)
		=
		C_{\lambda, n}^{\pm}
		|y|^{\lambda-n/2}
		H^{\pm}_{\lambda - n/2}(2\pi |y|).
		\]
		where $H^\pm_\nu$ denote the two Hankel functions $H_\nu^{(1)}$ and $H_\nu^{(2)}$, and
		$$C_{\lambda, n}^{\pm} = \pm i e^{\mp i\pi n/2} \frac{\pi^{\lambda+1}}{\Gamma(\lambda)}.$$ 
		For the particular exponent \(\alpha=(n+1)/2\), we thus have
		\[
		\mathcal F^{-1}
		\left((1-|x|^2\pm i0)^{-\alpha}\right)(y)
		=
		C_n |y|^{1/2}H^\pm_{1/2}(2\pi |y|),
		\]
		where, by the first half-integer Hankel formula,
		\[
		H^{(1)}_{1/2}(t)
		=
		-i\left(\frac{2}{\pi t}\right)^{1/2}e^{it},
		\qquad
		H^{(2)}_{1/2}(t)
		=
		i\left(\frac{2}{\pi t}\right)^{1/2}e^{-it}, \qquad t > 0.
		\]
		We conclude that indeed $T_\alpha \in \mathcal{F} L^\infty$.
	\end{proof}

	As a consequence of Lemma \ref{1/omega}, we have therefore established Helson's inequality
	$$\int_B\dfrac{|\hat{f}(x)|^2}{\omega_B(x)}dx\lesssim \|f\|_{L^1}^2,\quad f\in \PW^1(B).$$
	We wish to extend this result to other values of $q$, but we cannot apply the techniques of Section \ref{Sec4}, since we do not know whether the $\PW^p(B)$-spaces sit in an interpolation scale. To overcome this, we first sharpen the proof of Proposition~\ref{prop:symbol-ball}.
	\begin{lm} \label{lem:ballinterp}
		\label{lem:analytic-family-ball}
		There exists a family of distributions \(M_z\), analytic for
		\(0<\operatorname{Re}z<1\) and continuous for
		\(0\leq\operatorname{Re}z\leq1\), such that
		\[
		M_z|_B=\omega_B^{-z},
		\]
		and there is a constant \(C>0\) such that, for all \(t\in\mathbb R\),
		\[
		\|M_{it}\|_{L^\infty}
		\leq Ce^{C|t|},
		\qquad
		\|\mathcal F^{-1}M_{1+it}\|_{L^\infty}
		\leq Ce^{C|t|}.
		\]
	\end{lm}
	
	\begin{proof}
		We keep the notation from the proof of Proposition~\ref{prop:symbol-ball}, and additionally introduce a smooth compactly supported function $\widetilde\chi$  such that $\widetilde \chi \equiv 1$ on $\supp \chi$, and such that $a \geq C > 0$ on $\supp \widetilde\chi$. Define
		\[
		a_z
		=
		\widetilde\chi a^{z}.
		\]
		By a standard estimate, see for example \cite[Lemma 2.5]{bampouras2024besovspacesschattenclass}, we then have that
		$$\|\mathcal{F}^{-1} (\chi a_z) \|_{L^1}\leq C \sum_{\gamma\in\{0,1\}^n}\|\partial^\gamma (\chi a_z)\|_{L^\infty}\ \leq C(1+|\Image z|)^n.$$
		
		Next define
		\[
		b_z(\xi)
		=
		(1-\chi(\xi))\omega_B(\xi)^{-z},
		\]
		clearly satisfying
		\[
		\|b_z\|_{L^\infty}
		+
		\|b_z\|_{L^1}
		\leq C,
		\qquad 0\leq \operatorname{Re}z\leq1,
		\]
		and let $T_z$ be the distribution 
		\[
		T_z 
		=
		(1-|x|^2+i0)^{-\alpha z},
		\]
		defined as in Proposition~\ref{prop:symbol-ball}. The sought analytic family of distributions is given by
		\[
		M_z
		=
		b_z+\chi a_z T_z. \qquad 0 \leq \Real z \leq 1.
		\]
		On $B$ it holds that
		\[
		M_z
		=
		(1-\chi)\omega_B^{-z}
		+
		\chi a^z P^{-\alpha z}
		=
		\omega_B^{-z},
		\]	
		and so it remains to prove the boundary estimates.
		
		Consider \(z=1+it\).  The first term is harmless,
		\[
		\|\mathcal F^{-1}b_{1+it}\|_{L^\infty}
		\leq
		\|b_{1+it}\|_{L^1}
		\leq C.
		\]
		By the explicit formulas given in the previous proof, we have that
		\begin{align*}
			\mathcal F^{-1} T_{1+it}(y)
			&=
			ie^{-i\pi n/2} \frac{\pi^{ \alpha (1+it)+1}}{\Gamma(\alpha (1+it))}
			|y|^{\frac{1}{2} + i \alpha t}
			H^{(1)}_{\frac{1}{2} + i \alpha t}(2\pi |y|)  \\ =& e^{-i\pi n/2}
			\frac{e^{i2\pi |y|} \pi^{ \alpha (1+it)}  
				e^{\pi \alpha t/2}}{\Gamma(\alpha (1+it)) \Gamma(1 + i\alpha t)} |y|^{i\alpha t} 	\int_0^\infty
			e^{-s}s^{i\alpha t}
			\left(1+\frac{is}{4\pi |y|}\right)^{i\alpha t}\,ds,
		\end{align*}
		where we also applied a deformed version of the usual contour representation of the Hankel function, see \cite[Ch.~VII, \S 7.2, p.~196]{watson1944bessel}. From this formula it is immediate that
		\[
		\|\mathcal F^{-1}T_{1+it}\|_{L^\infty}
		\leq
		Ce^{C|t|}.
		\]
		In total we have thus proved that
		\[
		\|\mathcal F^{-1}M_{1+it}\|_{L^\infty}
		\leq
		\|\mathcal F^{-1}b_{1+it}\|_{L^\infty}
		+
		\|\mathcal F^{-1}(\chi a_{1+it})\|_{L^1}
		\|\mathcal F^{-1}T_{1+it}\|_{L^\infty}     \leq
		Ce^{C|t|}.
		\]
		
		When \(z=it\), $T_{it}$ is an ordinary function, which due to the branch point of the logarithm satisfies $\|T_{it}\|_{L^\infty}
		\leq e^{\pi\alpha |t|}.$ Thus 
		\[
		\|M_{it}\|_{L^\infty} \leq \|b_{it}\|_{L^\infty} + e^{\pi\alpha |t|} \|\chi a_{it}\|_{L^\infty} \leq C e^{\pi\alpha |t|}. \qedhere
		\]
	\end{proof}
	
	Using Lemma~\ref{lem:ballinterp} to facilitate complex interpolation we can prove the following theorem, which combined with Lemma \ref{bestd} gives Theorem \ref{ball}.
	
	\begin{theo}
		Let $1\leq q<2$ and $2\leq p\leq q'$, $p<\infty$. Then the boundary-weighted Fourier inequality
		$$\int_B\dfrac{|\hat{f}(x)|^p}{\omega_{B}(x)^{1+\frac{p}{q}-p}} \, dx\leq C\|f\|_{L^q}^p, \qquad f\in \PW^q(B),$$ holds for the exponent $d_c(p,q) = 1 + \frac{p}{q} - p$.
	\end{theo}
	\begin{proof}
		By Lemma \ref{1/omega} and Proposition \ref{prop:symbol-ball} we have already seen that Helson's inequality holds
		$$\int_B\dfrac{|\hat{f}(x)|^2}{\omega_{B}(x)}dx\leq C\|f\|_{L^1}^2, \quad f\in \PW^1(B).$$ 
		Thus, for $p>2$ and $ f\in \PW^1(B)$,
		$$\int_B\dfrac{|\hat{f}(x)|^p}{\omega_{B}(x)} \, dx\leq\|\hat{f}\|_{L^\infty}^{p-2}\int_B\dfrac{|\hat{f}(x)|^2}{\omega_{B}(x)} \, dx\leq C\|f\|_{L^1}^p.$$ 
		
		Now let us turn our attention to the endpoint $p=2$, for $1<q<2$, applying Lemma \ref{lem:analytic-family-ball}. Let $\mathcal{T}_z$ be the Fourier multiplier associated with $M_z$,
		$$\mathcal{T}_z f = \mathcal{F}^{-1} (M_z \hat{f}), \qquad 0\leq \Real z\leq 1,$$
		initially defined for Schwartz functions $f$. Then
		$$\|\mathcal{T}_{it} f\|_{L^2} \leq Ce^{C|t|}\|f\|_{L^2},$$
		and 
		$$\|\mathcal{T}_{1+it} f\|_{L^\infty} \leq Ce^{C|t|}\|f\|_{L^1}.$$ 
		By interpolation we therefore conclude for $\frac{1}{q_\theta}=\frac{1+\theta}{2},$ $\theta\in (0,1)$, that $\mathcal{T}_\theta$ extends to a bounded operator
		$\mathcal{T}_\theta \colon L^{q_\theta} \to L^{q_\theta'}$. In particular, $\mathcal{T}_{\frac{2}{q} -1} \colon L^q \to L^{q'}$ is bounded.
		Therefore, for smooth $f$ with compact Fourier support in $B$,
		$$\int_B\frac{|\hat{f}(x)|^2}{\omega_B^{\frac{2}{q}-1}(x)}dx= (M_{\frac{2}{q}-1} \hat{f}, \hat{f}) = (\mathcal{T}_{\frac{2}{q} -1} f, f) \leq C\|f\|_{L^q}^2.$$ 
		
		The case when $p=q'$ is just the boundedness of the Fourier transform. To finish the proof we must consider $1<q<2$ and $2<p < q'$.
		Let $\theta\in(0,1)$ be such that $p=2\theta+q'(1-\theta)$. Then H\"{o}lder's inequality and the boundedness of the Fourier transform yield
		\begin{eqnarray}
			\int_B\dfrac{|\hat{f}(x)|^p}{\omega_B^{1+\frac{p}{q}-p}(x)} \, dx&=& 	\int_B|\hat{f}(x)|^{q'(1-\theta)}\dfrac{|\hat{f}(x)|^{2\theta}}{\omega_B^{1+\frac{p}{q}-p}(x)} \, dx  \nonumber \\
			&\leq& \left(\int_B|\hat{f}(x)|^{q'}dx\right)^{1-\theta}\left(\int_B\dfrac{|\hat{f}(x)|^{2}}{\omega_B^{\frac{2}{q}-1}(x)}dx\right)^\theta \nonumber \\
			&\leq& C\|f\|_{L^q}^{q'(1-\theta)}\|f\|_{L^q}^{2\theta}=C\|f\|_{L^q}^p, \nonumber
		\end{eqnarray}
		which finishes the proof.
	\end{proof}

	\section{Kakeya sets and the restriction conjecture}\label{Sec6}
	
	In this section, we explore boundary-weighted Fourier inequalities further for the unit ball \(B \subset\mathbb R^n\), \(n\ge2\), observing and exploiting the connection with Kakeya-type phenomena, the Kakeya conjecture, and the restriction conjecture. 
	
	\subsection{The connection with Kakeya phenomena}
	
	Recall from the previous section that
	\begin{equation*}
		\omega_B(x)\simeq (1-|x|)^{(n+1)/2},
		\qquad \frac12<|x|<1.
	\end{equation*}
	With \(c_1=1/4\) and \(c_2=2\), we consider for \(N\ge2\) the frequency-side annuli
	\begin{equation*}
		A_N=\left\{ x \in B:
		\frac{c_1}{N^2}<1-|x|<\frac{c_2}{N^2}\right\}.
	\end{equation*}
	Since $\omega_B(x)\simeq N^{-(n+1)}$ for $x \in A_N$, the $(p,q,d)$ Fourier inequality \eqref{hardy} implies that
	\begin{equation}
		\|\widehat f\|_{L^p(A_N)}
		\lesssim N^{-(n+1)d/p}\|f\|_q,
		\qquad N \geq 2, \; f\in\PW^q(B).
		\label{eq:Hardyannular}
	\end{equation}
	In fact, as long as the necessary conditions \(1\le q\le2\) and \(p\le q'\) hold, summation over $N = 2^k$ gives a converse, which we record as a lemma.
	\begin{lm}
		Assume that $1 \leq q \leq 2$ and $p \leq q'$. Then the $(p,q,d)$ Fourier inequality \eqref{hardy} holds for every $d < d_c(p,q) = 1 + \frac{p}{q} - p$ if and only if \eqref{eq:Hardyannular} holds, with a constant depending only on $d$, for every $d < d_c(p,q)$.
	\end{lm}
	\begin{proof}
		We have seen one direction already. In the converse direction, assume that \eqref{eq:Hardyannular} holds for some $d$. Summing over $N = 2^k$ then gives the $(p, q, d')$ Fourier inequality for every $d' < d$; the boundedness of the Fourier transform is used to control the interior part. 
	\end{proof}
	
	Kakeya-type phenomena arise from overlapping thin tubes pointing in many directions on the space side; by Fourier duality, such collections can be encoded by functions with Fourier support localized to small frequency caps with separated directions. We will begin by describing how such configurations intertwine with boundary-weighted Fourier inequalities for $B$.
	
	Fix a nonzero $\phi \in S(\mathbb R^n)$ with Fourier transform supported in a small neighborhood of the origin.  Given $N \geq 2$ and a set
	$V\subset S^{n-1}$ of directions, pairwise separated by $N^{-1}$, pick corresponding rotations
	\(R_v\in SO(n)\) with \(R_ve_1=v\), and let
	\[
	z_{v,N}=\left(1-\frac{3}{2N^2}\right)v,
	\qquad
	Q_{v,N}= z_{v,N}+R_vS_N^{-1}\operatorname{supp}\widehat\phi,
	\]
	where $S_N$ is the diagonal matrix 	$S_N=\operatorname{diag}(N^2,N,\ldots,N)$.
	If the fixed support of \(\widehat\phi\) is chosen sufficiently small (independently of $N$ and $V$), the frequency caps $Q_{v,N}$ lie in $A_N$ and are pairwise disjoint.
	
	\begin{lm}
		\label{lem:capinequality}
		Let $1\le q <2$, $d\in\mathbb R$, and assume that the $(q, q, d)$ Fourier inequality holds for $B$.
		Given \(N\ge2\), an
		\(N^{-1}\)-separated set \(V\subset S^{n-1}\), arbitrary points
		\(a_v\in\mathbb R^n\), and \(L\ge1\), define
		\begin{equation*}
			\varphi_{v,a_v}(y)
			= J_N^{-1} e^{2\pi i \langle y, z_{v,N}\rangle}
			\phi\bigl(S_N^{-1}R_v^{-1}(y-a_v)\bigr), \qquad v \in V,
		\end{equation*}
		and
		\[
		K_L=\bigcup_{v\in V}
		\left(a_v+R_vS_N[-L/2,L/2]^n\right),
		\]
		where $J_N = \det S_N = N^{n+1}$.
		Then, for every \(M>0\),
		\begin{equation} \label{eq:maincapconcl}
			J_N^{d-1} \# V
			\lesssim
			|K_L|^{1-q/2}\bigl(J_N^{-1} \# V \bigr)^{q/2}
			+C_ML^{-M}J_N^{1-q} \# V,
		\end{equation}
		where $C_M$ depends only on $M$ and $\phi$, and the remaining implied constant is independent of $M, N,V,(a_v),$ and $L$.
	\end{lm}
	
	\begin{proof}
		We have that
		\[
		\widehat{\varphi}_{v,a_v}(x)
		= e^{-2\pi i \langle x- z_{v,N}, a_v\rangle} \widehat\phi\bigl(S_NR_v^{-1}(x- z_{v,N})\bigr),
		\]
		and, consequently,
		\begin{equation*}
			\operatorname{supp}\widehat{\varphi}_{v,a_v}\subset Q_{v,N},
			\qquad
			\|\widehat{\varphi}_{v,a_v}\|_q^q\simeq J_N^{-1},
			\qquad
			\|\varphi_{v,a_v}\|_2^2\simeq J_N^{-1}.
		\end{equation*}
		For signs $\epsilon = (\epsilon_v)_{v \in V} \in \{-1,1\}^V$, let
		\(g_\epsilon=\sum_{v\in V}\epsilon_v\varphi_{v,a_v}\).  Then, since the caps $Q_{v,N}$ are disjoint,
		\[
		J_N^{d-1} \# V \approx \int_B|\widehat g_\epsilon(x)|^q \omega_B(x)^{-d}\,d x \lesssim \|g_\epsilon\|_{L^q}^q,
		\]
		where we have also applied the assumed $(q, q, d)$ Fourier inequality. Khintchine's inequality, applied pointwise and integrated, therefore gives us that
		\begin{equation*}
			J_N^{d-1} \#V
			\lesssim
			\int_{\mathbb R^n}
			\left(\sum_{v\in V}|\varphi_{v,a_v}(x)|^2\right)^{q/2}\,dx.
		\end{equation*}
		On \(K_L\), H\"older's inequality yields
		\begin{align*}
			\int_{K_L}
			\left(\sum_{v\in V}|\varphi_{v,a_v}(x)|^2\right)^{q/2}\,dx
			&\le |K_L|^{1-q/2}
			\left(\sum_{v\in V}\|\varphi_{v,a_v}\|_2^2\right)^{q/2}\\
			&\approx |K_L|^{1-q/2}
			\bigl(J_N^{-1} \# V \bigr)^{q/2}.
		\end{align*}
		On \(K_L^c\), we simply apply the subadditivity of the \(q/2\)-power and, after a change of variable, the rapid decay
		of \(\phi\),
		\begin{align*}
			\int_{K_L^c}
			\left(\sum_{v\in V}|\varphi_{v,a_v}(x)|^2\right)^{q/2}\,dx
			&\le \sum_{v\in V}\int_{K_L^c}|\varphi_{v,a_v}(x)|^q\,dx\\
			&\lesssim_M L^{-M} J_N^{1-q} \# V.
		\end{align*}
		Collecting these bounds proves the lemma.
	\end{proof}
	By a width-one \textit{tube} with centre $b_v \in \R^n$ and direction $v \in S^{n-1}$, we shall mean a set of the form 
	\begin{equation} \label{eq:tubedef}
		T_v=b_v+R_v\bigl([-N/2,N/2]\times[-1/2,1/2]^{n-1}\bigr).
	\end{equation}
	Under the assumption of a valid $(q,q,d)$-inequality, Lemma~\ref{lem:capinequality} forces the following lower estimate for the size of unions of such tubes.
	\begin{lm}
		\label{lem:tubeinequality}
		Let \(1\le q<2\), and assume that the $(q,q,d)$ Fourier inequality holds for some $d \leq d_c(q,q) = 2 - q$.
		Given an \(N^{-1}\)-separated set \(V\subset S^{n-1}\), $N \geq 2$, with corresponding tubes \eqref{eq:tubedef}, let $U = \bigcup_{v \in V} T_v$. If $d < 2 - q$, then, for every $\epsilon > 0$, it must hold that
		\begin{equation} 		 	\label{eq:subcriticaltubeunion}
			|U|\gtrsim_{\epsilon}
			N^{1- \frac{(n+1)(2-q - d)}{1 - q/2} -\epsilon} \# V.
		\end{equation}
		If instead $d = 2-q$, then it must be that 
		\begin{equation}		 	\label{eq:criticaltubekakeya}
			|U|\gtrsim N \#V.
		\end{equation}
		The implied constants are uniform in $N,V,$ $(b_v)$, and the choice of rotations \((R_v)\).
	\end{lm}
	
	\begin{proof}
		We apply Lemma~\ref{lem:capinequality} with $a_v=NLb_v$, so that $|K_L|=(NL)^n|U|$. If $d < 2 -q$, let $L = A N^{\epsilon'}$, for $\epsilon' > 0$ and a large constant $A$. Then the quotient of the second term in the right-hand side of \eqref{eq:maincapconcl} and the left-hand side is
		\[
		\frac{C_ML^{-M}J_N^{1-q} \# V}{J_N^{d-1} \# V} =C_M A^{-M}  N^{-M\epsilon'} N^{(n+1)(2-q - d)}.
		\]
		Therefore, choosing $M$ so large that $M\epsilon' \geq (n+1)(2-q-d)$, and then $A$ sufficiently large, we can make this quotient as small as needed, uniformly in $N$ and $V$.  Therefore \eqref{eq:maincapconcl} reads 
		\[
		J_N^{d-1} \# V
		\lesssim
		|K_L|^{1-q/2}\bigl(J_N^{-1} \# V \bigr)^{q/2},
		\]
		which upon rearrangement is 
		\[
		|U|
		\gtrsim
		A^{-n}N^{1-\frac{(n+1)(2-q - d)}{1 - q/2}-n\epsilon'}\#V.
		\]
		When $d = 2-q$, we apply the exact same argument with $\epsilon' = 0$.
	\end{proof}
	
	The reader familiar with Besicovitch/Kakeya sets recognizes that the tubes $T_v$ can be chosen in contradiction with \eqref{eq:criticaltubekakeya}.
	
	\begin{theo} \label{Hardygate}
		For every \(1\le q<2\), the $(q, q, d_c)$ boundary-weighted Fourier inequality fails for the critical exponent $d_c = 2-q$. 
	\end{theo}
	
	\begin{proof}
		Suppose first that $n = 2$.  Let \(E\subset\mathbb R^2\) be a compact planar Besicovitch set of measure
		zero; given for example by the Perron tree construction \cite[Chapter~VIII]{MR596037}. For a small $\delta > 0$, let $E_\delta = E + B(0,\delta)$ be the thickening of $E$. For each $N$, choose an $N^{-1}$-separated set $V \subset S^1$, consisting of $N$ directions. For every $v \in V$, there is a length one segment $I_v \subset E$ with direction $v$. Thus, there is a width-one tube $T_v$ of the form \eqref{eq:tubedef} contained in the rescaled thickened set $N E_{c/N}$, where $c$ is a universal constant. 
		If the $(q, q, d_c)$-inequality were true, inequality \eqref{eq:criticaltubekakeya} of Lemma~\ref{lem:tubeinequality} would thus imply that
		\[N^2|E_{c/N}| \gtrsim N^2.
		\]
		This of course contradicts the fact that $|E_{c/N}| \to 0$, as $N \to \infty$.
		
		When $n > 2$, we run the same argument with the set $E' = E \times [-1/2,1/2]^{n-2}$ and the tubes $T_v' = T_v \times [-1/2,1/2]^{n-2}$.
	\end{proof}
	
	We can also prove that if the subcritical boundary-weighted Fourier inequalities hold, then every compact Kakeya set in $\R^n$ must have full Minkowski dimension $n$, thus verifying the Kakeya conjecture in its Minkowski form. See, for example, \cite[Chapter~5]{Mattila1995} for the relevant definitions and background material. 
	
	\begin{theo}
		Let \(1\le q<2\).  Assume that the $(q,q,d)$ boundary-weighted Fourier inequality for $B$ holds for every $d < d_c(q,q) = 2-q$, 
		\begin{equation*} 
			\int_{B}\dfrac{|\hat{f}(x)|^q}{\omega_{B}^d(x)}dx\leq C_{q, d}\|f\|_{L^q}^q,\quad f\in \PW^q(B). 
		\end{equation*}
		Then every compact Kakeya set in
		\(\mathbb R^n\) has full lower Minkowski dimension $n$.
	\end{theo}
	\begin{proof}
		Given a compact Kakeya set $E$, we must prove that $|E_\delta| \gtrsim_{\epsilon} \delta^{\epsilon}$ for every $\epsilon > 0$. Choose an $N^{-1}$-separated set $V \subset S^{n-1}$, consisting of $cN^{n-1}$ directions. As before, there is a length one segment $I_v \subset E$ with direction $v$, and a corresponding width-one tube $T_v$ of length $N$ contained in $N E_{c'/N}$. Applying the conclusion \eqref{eq:subcriticaltubeunion} of Lemma~\ref{lem:tubeinequality} with $d$ sufficiently close to $2-q$ then yields
		\[N^n |E_{c/N}| \gtrsim_{\epsilon} N^{n - \epsilon},\]
		as desired.
	\end{proof}

	\subsection{Consequences of the restriction conjecture}
	Let \(R(p, q)\) denote the spherical restriction estimate
	\[
	\|\widehat F|_{S^{n-1}}\|_{L^p(S^{n-1})}
	\lesssim \|F\|_{L^q(\mathbb R^n)},
	\qquad F\in\mathcal S(\mathbb R^n).
	\]
	The spherical restriction conjecture \cite{Tao2004restriction} asserts that $R(p,q)$
	holds whenever
	\begin{equation}
		1\le q<\rho_n:=\frac{2n}{n+1},
		\qquad
		1\le p \le \kappa_n(q):=\frac{(n-1)q'}{n+1}.
		\label{eq:restrictionconj}
	\end{equation}
	The critical case \(p=\kappa_n(q)\) gives the so-called
	Knapp line,
	\begin{equation*}
		\frac{n+1}{q}+\frac{n-1}{p}=n+1.
	\end{equation*}
	In this subsection, our goal is to show that the restriction conjecture implies a range of boundary-weighted Fourier inequalities for $B$. In fact, it implies a stronger version, where the functions need not have Fourier support in $B$.
	
	\begin{lm}
		\label{lem:resttohardy}
		Suppose that $1\le q\le2$, $1\le p <\infty$, $p \le q'$, and that
		$R(p,q)$ holds.  Then, for every
		$d< 2 - \rho_n = \frac{2}{n+1}$,
		\begin{equation*}
			\int_B \frac{|\widehat F(x)|^p}
			{\omega_B(x)^{d}} \,d x
			\lesssim \|F\|_{L^q}^p,
			\qquad F\in\mathcal S(\mathbb R^n).
		\end{equation*}
	\end{lm}
	\begin{remark}
		On the Knapp line $p = \kappa_n(q)$, it holds that $2 - \rho_n = d_c(p,q) = 1 + \frac{p}{q} - p$, and thus Lemma~\ref{lem:resttohardy} already gives the full range $d < d_c(p,q)$ in this case.
	\end{remark}
	
	\begin{proof}
		Dilation of the restriction estimate gives, 
		\[
		\|\widehat F(r\,\cdot)\|_{L^p(S^{n-1})}
		\lesssim \|F\|_q,
		\]
		uniformly for
		\(1/2\le r\le1\).
		Polar coordinates thus give us that
		\[
		\int_{B \setminus B(0,\frac12)} \frac{|\widehat F(x)|^p}
		{\omega_B(x)^{d}} \,d x \lesssim \|F\|_q^p
		\int_{1/2}^1(1-r)^{-\frac{n+1}{2}d} \,dr.
		\]
		The integral is finite exactly when $d<\frac{2}{n+1}$.  On $B(0,\frac{1}{2})$ we simply use the boundedness of the Fourier transform from $L^q$ to $L^{q'}$, and the fact that $p \leq q'$.
	\end{proof}
	
	We obtain the main result of this subsection by interpolation.
	\begin{theo}
		\label{thm:restrconj}
		Assume that the spherical restriction conjecture
		\eqref{eq:restrictionconj} holds.  Then, if $\rho_n\le  q< 2$ and $d<2-q$, 
		\begin{equation}
			\int_B \frac{|\widehat F(x)|^q}
			{\omega_B(x)^{d}} \,d x
			\lesssim \|F\|_{L^q}^q,
			\qquad F\in\mathcal S(\mathbb R^n).
			\label{eq:condcrit}
		\end{equation}
		In particular, the $(q, q, d)$ boundary-weighted Fourier inequality holds for $B$ for every $d < d_c(q,q) = 2-q$. 
	\end{theo}
	
	\begin{proof}
		We first consider the case when $q = \rho_n$, $d < 2 - \rho_n$. 
		Choose \(1<q_-<\rho_n<q_+<2\).  Then $R(q_-, q_-)$ holds by assumption, since 
		\(q_-\le\kappa_n(q_-)\). Lemma~\ref{lem:resttohardy}
		therefore gives, 
		\begin{equation*}
			\int_B \frac{|\widehat F(x)|^{q_-}}
			{\omega_B(x)^{d_-}} \,d x
			\lesssim \|F\|_{L^{q_-}}^{q_-},
			\qquad F\in\mathcal S(\mathbb R^n), \; \; d_- < 2 - \rho_n.
		\end{equation*}
		On the other hand, the boundedness of the Fourier transform immediately yields
		\begin{equation*}
			\|\widehat F\|_{L^{q_+}(B)}\lesssim\|F\|_{q_+}.
		\end{equation*}
		Complex interpolation of these two inequalities now gives \eqref{eq:condcrit} for $q = \rho_n$ and 
		\[
		d =  \frac{\theta \rho_n}{q_-}d_{-},
		\]
		where $\theta$ is such that $\frac1{\rho_n}=\frac\theta{q_-}+\frac{1-\theta}{q_+}$. By keeping $q_+$ fixed and letting $q_{-} \to \rho_n$, we thus obtain \eqref{eq:condcrit} for $q = \rho_n$ and any $d < d_-$. Since $d_- < 2 - \rho_n$ was arbitrary, this in fact gives us the full range $d < 2 - \rho_n$.
		
		Finally, to obtain the full statement of the Theorem, we interpolate between the just established case and Plancherel.
	\end{proof}
	
	Zygmund proved the restriction conjecture for the circle
	\cite{Zygmund1974}, and we therefore obtain a sharp unconditional result for the disc.
	\begin{cor}
		\label{cor:zygmund}
		Unconditionally, the $(q,q,d)$ boundary-weighted Fourier inequality 
		\[
		\int_{\mathbb{D}} \frac{|\hat{f}(x)|^q}{(1-|x|)^{\frac{3}{2}d}} \, dx \lesssim \|f\|_{L^q}^q, \qquad f \in \PW^q(\mathbb{D}),
		\]
		holds for the unit disc $\mathbb{D} \subset \R^2$, when $\frac43 \leq q <2$ and $d < d_c(q,q) = 2-q$.  
	\end{cor}

	\section{Boundary-weighted Fourier inequalities and Hankel operators}\label{Sec7}
	
	\subsection{Schatten class Hankel operators}
	Boundary-weighted Fourier inequalities originated from truncated Hankel operators $\Ha_\phi:\PW^2(\Omega)\to \PW^2(\Omega)$. For a distribution $\hat{\phi}$ in $2\Omega$, the Hankel operator is defined via the formula
	$$\widehat{\Ha}_\phi f(x)=\int_{\mathbb{R}^n}\widehat{\phi}(x+y)\hat{f}(y) \, dy \quad x\in\Omega,$$
	initially for $f \in \mathcal{F}^{-1}C_c^\infty(\Omega)$. We refer to \cite{bampouras2024besovspacesschattenclass} for further details and bibliography, but only mention that the prominent question in higher dimensions, open already in the simplest case $\Omega = \R_+^n$, is whether the conclusion of Nehari's theorem is valid: given a bounded Hankel operator $\Ha_\phi:\PW^2(\Omega)\to \PW^2(\Omega)$, does there exist $\psi \in L^\infty$ such that the restriction $\hat{\psi}|_{2\Omega}$ coincides with $\hat{\phi}$ as a distribution in $2\Omega$? In this case, we say that $\phi$ has a \textit{bounded symbol}.
	
	To illustrate the connection with boundary-weighted Fourier inequalities, consider first the Hilbert--Schmidt class of Hankel operators. Then $$\|\Ha_\phi\|_{S^2}^2 = \int_\Omega\int_\Omega |\widehat{\phi}(x+y)|^2 \,dx \, dy=2^{-n}\int_{2\Omega}|\widehat{\phi}(x)|^2\omega_{2\Omega}(x) \, dx.$$ Therefore a simple duality argument, explained in the proofs of Propositions \ref{firstprop} and \ref{prop:Nehari-implies-Hardy}, shows that Helson's inequality, the $(2,1,1)$ Fourier inequality for $\Omega$,
	$$\int_\Omega\dfrac{|\hat{f}(x)|^2}{\omega_\Omega(x)}dx\lesssim \|f\|_{L^1}^2, \quad f\in \PW^1(\Omega)$$ is equivalent to the following statement: every Hilbert-Schmidt Hankel operator on $\PW(\Omega)$ has a bounded symbol. We say that \textit{Nehari's theorem holds for $S^2(\PW^2(\Omega))$}. 
	
	Helson's inequality for the ball was demonstrated in Section~\ref{Sec5}, and accordingly we have the following reinterpretation. We refer to \cite{MR878246} for the case $n=1$.
	
	\begin{cor} \label{cor:S2Nehari}
		In the case of the unit ball $B \subset \R^n$, $n \geq 2$, Nehari's theorem holds for $S^2(\PW^2(B))$. 
	\end{cor}
	
	This result sits at the intersection of two general implications, which this subsection is devoted to. More precisely, we are going to show that when $p\geq 2$, the $(p, 1, 1)$ boundary-weighted Fourier inequality \eqref{hardy} for $\Omega$ always implies Nehari's theorem for $S^{p'}(\PW^2(\Omega))$ and, assuming a Besov space characterization, that Nehari's theorem for $S^p(\PW^2(\Omega))$ implies the corresponding $(p', 1, 1)$ Fourier inequality for $\Omega$. Here $\frac{1}{p} + \frac{1}{p'} = 1$.
	
	In order to prove these implications we will need to make a connection between the $S^p$-Nehari theorem and the Besov spaces of Paley--Wiener type which were introduced in Section~\ref{subsec:besov}. Since we proved that every set $\Omega$ is admissible in Theorem~\ref{thm:main}, the necessity result \cite[Theorem 1.4]{bampouras2024besovspacesschattenclass} implies that we always have the estimate
	\begin{equation} \label{eq:besovnec}
		\|\phi\|_{B_{p,p}^{1/p}(2\Omega)} \leq C \|\Ha_\phi\|_{S^p(\PW^2(\Omega))}, \qquad 1 \leq p \leq \infty.
	\end{equation}
	This can be dualized as follows.
	\begin{lm}\label{embednehari}
		Let $1\leq p\leq \infty$. If $$\PW^1(\Omega)\hookrightarrow B_{p',p'}^{-1/p}(\Omega),$$ then Nehari's theorem holds for $S^p(\PW^2(\Omega))$.
	\end{lm}
	\begin{proof}
		Suppose that $\Ha_\phi\in S^p$ and that $f \in \mathcal{F}^{-1} C_c^\infty(2\Omega)$. Then, by duality of Besov spaces \cite[Lemma 2.7]{bampouras2024besovspacesschattenclass}, \eqref{eq:besovnec}, and the hypothesis (rescaled to $2\Omega$), we have that
		$$|\langle \phi,f\rangle|\lesssim \|\phi\|_{B_{p,p}^{1/p}(2\Omega)}  \|f\|_{B_{p',p'}^{-1/p}(2\Omega)}\lesssim \|\Ha_\phi\|_{S^p}\|f\|_{L^1}.$$
		Thus by density, $\phi$ generates a bounded functional on $\PW^1(2\Omega) \subset L^1$. By Hahn-Banach there is $\psi\in L^\infty$ such that for $f\in \PW^1(2\Omega)$, $\langle f,\psi\rangle  = \langle f,\phi\rangle$. Therefore $\hat{\phi}=\hat{\psi}$ in $2\Omega$, hence $\Ha_\phi=\Ha_\psi$.
	\end{proof}
	In particular, Lemma \ref{embednehari} clearly applies when $p = 1$. Indeed, $\PW^1(\Omega)\hookrightarrow B_{\infty,\infty}^{-1}(\Omega)$, since 
	$$a^{-j}\|f\ast\phi_j^i\|_{L^\infty}\leq a^{-j}\|f\|_{L^1}\|\phi_j^i\|_{L^\infty} \lesssim \|f\|_{L^1},$$
	by Lemma~\ref{estim}.
	
	Now, let us prove the first result of this subsection.
	
	\begin{prop}\label{firstprop}
		Let $2 \leq p < \infty$. If the $(p,1,1)$ boundary-weighted Fourier inequality holds,
		$$\int_\Omega\dfrac{|\hat{f}(x)|^{p}}{\omega_\Omega(x)}dx\lesssim \|f\|_{L^1}^{p},\quad f\in\PW^1(\Omega),$$ 
		then Nehari's theorem holds for $S^{p'}(\PW^2(\Omega))$ for $\Omega$.
	\end{prop}
	\begin{proof}
		We may assume that the Fourier inequality holds for $2\Omega$, since it is unaffected by affine bijections. Let $\{\phi_j^i\}\in\mathscr{P}(2A_j^i)$,  where $\{A_j^i\}$ is an admissible cover for $\Omega$. For $f\in \mathcal{F}^{-1} C_c^\infty(2\Omega)$, using the boundedness of the Fourier transform, H\"{o}lder's inequality in the integral, and \ref{Axiom5}, we can estimate
		\begin{eqnarray} \|f\|_{B_{p,p}^{-1/p'}(2\Omega)}^p&=&\sum_{j,i}a^{-jp/p'}\|f\ast\phi_j^i\|_{L^p}^p\leq \sum_{j,i}a^{-jp/p'}\left(\int_{2A_j^i}|\hat{f}(x)\widehat{\phi}_j^i(x)|^{p'}dx\right)^{\frac{p}{p'}} \nonumber \\
			&\lesssim& \sum_{j,i}a^{-jp/p'} a^{jp(\frac{1}{p'}-\frac{1}{p})}\int_{2A_j^i}|\hat{f}(x)\widehat{\phi}_j^i(x)|^{p}dx \nonumber \\ &=& \sum_{j,i}  a^{-j}\int_{2A_j^i}|\hat{f}(x)\widehat{\phi}_j^i(x)|^{p}dx \lesssim \int_{2\Omega}\frac{|\hat{f}(x)|^{p}}{\omega_{2\Omega}(x)}dx\lesssim \|f\|_{L^1}^p. \nonumber
		\end{eqnarray}
		The result follows by Lemma \ref{embednehari}.
	\end{proof} 
	
	If we had access also to the reverse inequality of \eqref{eq:besovnec}, we would obtain that the class of symbols generating Schatten class Hankel operators $\Ha_\phi : \PW^2(\Omega) \to \PW^2(\Omega)$ coincides exactly with $B_{p,p}^{1/p}(2\Omega)$,
	\begin{equation} \label{eq:besovsuf}
		\|\phi\|_{B_{p,p}^{1/p}(2\Omega)} \approx \|\Ha_\phi\|_{S^p(\PW^2(\Omega))}.
	\end{equation}
	By admissibility of $\Omega$, Theorem~\ref{thm:main}, and \cite[Theorem~1.3]{bampouras2024besovspacesschattenclass}, we know that \eqref{eq:besovsuf} is always holds true for $1 \leq p \leq 2$. Furthermore, we know that \eqref{eq:besovsuf} holds for all $1 \leq p < \infty$ when $\Omega = P$ is a simple polytope, and for all $1 \leq p < 2 \frac{n+1}{n-1}$ when $\Omega = B \subset \R^n$ is the unit ball, see \cite{bampouras2024besovspacesschattenclass}. We expect that \eqref{eq:besovsuf} holds for all $1 \leq p < \infty$ for arbitrary convex sets $\Omega$.
	
	\begin{prop}\label{prop:Nehari-implies-Hardy}
		Let $2\leq p< \infty$. Assume that 
		$$\|\Ha_\phi\|_{S^p}\lesssim \|\phi\|_{B_{p,p}^{1/p}(2\Omega)}$$
		holds. Then Nehari's theorem for $S^p(\PW^2(\Omega))$ implies the $(p', 1, 1)$ boundary-weighted Fourier inequality for $\Omega$,
		$$\int_\Omega\dfrac{|\hat{f}(x)|^{p'}}{\omega_\Omega(x)}dx\lesssim \|f\|_{L^1}^{p'},\quad f\in\PW^1(\Omega).$$
	\end{prop}
	\begin{proof}
		Let $\{\phi_j^i\}\in\mathscr{P}(2A_j^i),$ where $\{A_j^i\}$ is an admissible cover of $\Omega$. As in the proof of Proposition~\ref{firstprop}, we can bound
		\begin{eqnarray} \|\phi\|_{B_{p,p}^{1/p}(2\Omega)}^p&=&\sum_{j,i}a^j\|\phi\ast\phi_j^i\|_{L^p}^p\leq \sum_{j,i}a^j\left(\int_{2A_j^i}|\widehat{\phi}(x)\widehat{\phi}_j^i(x)|^{p'}dx\right)^{\frac{p}{p'}} \nonumber \\
			&\lesssim& \sum_{j,i}a^j a^{jp(\frac{1}{p'}-\frac{1}{p})}\int_{2A_j^i}|\widehat{\phi}(x)\widehat{\phi}_j^i(x)|^{p}dx \lesssim\int_{2\Omega}|\widehat{\phi}(x)|^{p}\omega_{2\Omega}^{p-1}(x)dx. \nonumber
		\end{eqnarray}
		By an argument with the closed graph theorem, we can restate the assumption that Nehari's theorem holds for $S^p$ as follows: for every Hankel operator $\Ha_\phi\in S^p(\PW(\Omega))$ there is a bounded function $\psi\in L^\infty(\mathbb{R}^n)$ such that $\widehat{\phi}=\widehat{\psi}$ in $2\Omega$ and  $\|\psi\|_{L^\infty}\leq C\|\Ha_\phi\|_{S^p}$. Therefore this symbol $\psi = \psi_{\phi}$ satisfies
		$$\|\psi\|_{L^\infty}\lesssim \left(\int_{2\Omega}|\widehat{\phi}(x)|^{p}\omega_{2\Omega}^{p-1}(x)dx\right)^{1/p}.$$
		Given $f\in\mathcal{F}^{-1} C_c^\infty(2\Omega)$, consider the symbol $\phi \in L^\infty$ given by 
		$$\widehat{\phi}(x)=\dfrac{|\hat{f}(x)|^{p'-2}\hat{f}(x)}{\omega_{2\Omega}(x)} \chi_{ \{x \in 2\Omega  \, : \, |\hat{f}(x)| > 0\}}.$$ Then we have that
		\[\left(\int_{2\Omega}|\widehat{\phi}(x)|^{p}\omega_{2\Omega}^{p-1}(x)dx\right)^{1/p} = \left(\int_{2\Omega} \dfrac{|\hat{f}(x)|^{p'}}{\omega_{2\Omega}(x)}dx\right)^{1/p} < \infty,
		\]
		and thus, by the preceding estimates, $\Ha_\phi \in S^p$, and there is $\psi = \psi_\phi \in L^\infty$ such that $\widehat{\psi} = \hat{\phi}$ in $2\Omega$ and
		\[
		\|\psi\|_{L^\infty} \lesssim \left(\int_{2\Omega} \dfrac{|\hat{f}(x)|^{p'}}{\omega_{2\Omega}(x)}dx\right)^{1/p}.
		\]
		On the other hand,
		\[
		\int_{2\Omega} \dfrac{|\hat{f}(x)|^{p'}}{\omega_{2\Omega}(x)}dx = \langle f,\psi\rangle\ \leq \|f\|_{L^1}\|\psi\|_{L^\infty}\lesssim \|f\|_{L^1}\left(\int_{2\Omega} \dfrac{|\hat{f}(x)|^{p'}}{\omega_{2\Omega}(x)}dx\right)^{1/p}. 
		\]
		This is the inequality we needed to show.
	\end{proof}
	
	\subsection{The Hilbert matrix and the Hardy inequality}
	
	We finish by discussing the endpoint $(p,q,d) = (1,1,1)$, which yields the Paley--Wiener analogues of both the classical Hardy inequality and the Hilbert matrix. 
	
	For a function $f(z) = \sum_{k \geq 0} \hat{f}(k) z^k \in H^1(\mathbb{D})$, the classical inequalities of Hardy and Carleman are
	\begin{equation*}
		\sum_{k\geq0}\frac{|\hat{f}(k)|}{k+1} \leq \pi \|f\|_{H^1},
		\qquad
		\sum_{k\geq0}\frac{|\hat{f}(k)|^2}{k+1} \leq \|f\|_{H^1}^2.
	\end{equation*}
	The Carleman inequality is the one-variable model for the $(2,1,1)$ Helson inequality considered in
	this paper; the analogy is exactly motivated by the equivalence of Helson's inequality and the $S^2$-Nehari theorem of Corollary~\ref{cor:S2Nehari}. The same sequence $((k+1)^{-1})_{k\geq0}$ also generates the classical Hilbert matrix
	\[
	H = \left(\frac{1}{m+n+1}\right)_{m,n\geq0}.
	\]
	This is not a coincidence. The quadratic form of the Hilbert matrix makes a direct connection with both the Hardy and the Carleman inequality:
	\[
	\sum_{m,n\geq0}\frac{\hat{f}(m) \hat{g}(n)}{m+n+1}
	=
	\sum_{k\geq0}\frac{\widehat{fg}(k)}{k+1}, \qquad f, g \in H^2(\mathbb{D}).
	\]
	
	In our setting, for a general set $\Omega$ the analogous $(1,1,1)$ Hardy-type inequality and the $(2,1,1)$ Helson inequality are, for $f\in\PW^{1}(\Omega)$,
	\[
	\int_\Omega \dfrac{|\hat{f}(x)|}{\omega_{\Omega}(x)}dx\leq C\|f\|_{L^1}, \qquad \int_\Omega \dfrac{|\hat{f}(x)|^2}{\omega_{\Omega}(x)}dx\leq C'\|f\|_{L^1}^2.
	\]
	The analogous operator to the Hilbert matrix is the integral operator $\mathcal{H}_\Omega : L^2(\Omega) \to L^2(\Omega)$ given by
	\[
	\mathcal{H}_\Omega(f)(x)=\int_{\Omega}\dfrac{f(y)}{\omega_{2\Omega}(x+y)} \, dy, \qquad f \in L^2(\Omega), \; x\in \Omega.
	\]
	The Hardy inequality, the Helson inequality, and the Hilbert matrix are again connected via the quadratic form of the Hilbert matrix,
	\begin{equation} \label{eq:hilbquad}
		\langle \mathcal{H}_\Omega \hat{f}, \overline{\hat{g}} \rangle = \int_{\Omega} \int_{\Omega}\dfrac{\hat{f}(x)\hat{g}(y)}{\omega_{2\Omega}(x+y)}\, dx   \, dy = \int_{2\Omega}\dfrac{\widehat{fg}(x)}{\omega_{2\Omega}(x)} \, dx, \qquad f,g \in \PW^2(\Omega).
	\end{equation}
	
	For the ball, $\Omega = B \subset \R^n$, $n \geq 2$, we have already shown that $\mathcal{H}_B$ is bounded. Indeed, in Proposition~\ref{prop:symbol-ball} we produced a symbol $\psi \in L^\infty$ such that $\hat{\psi}$ coincides with $\omega_B^{-1}$ on $B$ -- this is exactly how we proved Helson's inequality for the ball. On the other hand, in Theorem~\ref{Hardygate} we showed that the $(1,1,1)$ Hardy inequality fails for $B$. 
	
	For a general set $\Omega$, we of course have the implications $(1)\Rightarrow (2) \Rightarrow (3)$ among these properties:
	\begin{enumerate}
		\item Hardy's inequality holds, that is, the $(1,1,1)$ boundary-weighted Fourier inequality holds for $\Omega$.
		\item There is $\psi \in L^\infty$ such that $\hat{\psi}|_{\Omega} = \omega_\Omega^{-1}$ as distributions in $\Omega$.
		\item The Hilbert matrix $\mathcal{H}_\Omega \colon L^2(\Omega) \to L^2(\Omega)$ is bounded.
	\end{enumerate}
	
	The next lemma shows that, under Nehari's theorem, these properties are equivalent.
	\begin{lm} \label{lem:hardyineqandhilbert}
		Suppose that the full Nehari theorem holds: every bounded Hankel operator $\Ha_\phi \colon \PW^2(\Omega) \to \PW^2(\Omega)$ is generated by a bounded symbol. Then the $(1,1,1)$ Hardy inequality holds for $\Omega$ if and only if $\mathcal{H}_\Omega \colon L^2(\Omega) \to L^2(\Omega)$ is bounded.
	\end{lm}
	\begin{proof}
		The Hardy inequality immediately implies that the Hilbert matrix is bounded, by \eqref{eq:hilbquad}. For the other direction suppose that $\mathcal{H}_\Omega$ is bounded. The key is that Nehari's theorem is equivalent to a weak factorization statement, see \cite[Proposition~5.1]{MR4227573}. More precisely, Nehari's theorem holds if and only if there is $C>0$ such that every $f\in \PW^1(2\Omega)$ can be written as a sum $f=\sum_j g_j h_j$, where $g_j,h_j\in \PW^2(\Omega)$ and $\|f\|_{L^1} \leq \sum_j \|g_j\|_{L^2}\|h_j\|_{L^2}\leq C\|f\|_{L^1}$. Therefore, if we introduce new functions $G_j, H_j \in \PW^2(\Omega)$ such that $\widehat{G_j}=|\hat{g}_j|$ and $\widehat{H_j}=|\hat{h}_j|$, we have that
		\[
		|\hat{f}| = \bigl| \sum_j \hat{g}_j * \hat{h}_j \bigr| \leq \sum_j \widehat{G_j H_j}.
		\]
		Therefore, applying that the quadratic form \eqref{eq:hilbquad} of $\mathcal{H}_\Omega$ is bounded, we obtain
		$$\int_{2\Omega}\dfrac{|\hat{f}(x)|}{\omega_{2\Omega}(x)}dx\leq \sum_j\int_{2\Omega}\dfrac{\widehat{G_j H_j}(x)}{\omega_{2\Omega}(x)}dx \lesssim \sum_j \|G_j\|_{L^2}\|H_j\|_{L^2}\lesssim \|f\|_{L^1}.$$
		That is, the Hardy inequality holds.
	\end{proof}
	
	In particular, we have arrived at yet another proof that Nehari's theorem must fail for the ball, since Hardy's inequality fails, but $\mathcal{H}_B$ is bounded. See \cite{MR4502777, MR4700194} for other proofs of this result. 
	
	\begin{cor}
		Nehari's theorem fails for balls in $\mathbb{R}^n$, $n\geq 2$.
	\end{cor}

	\bibliographystyle{plain}
	\bibliography{Helson}
	
\end{document}